\documentclass{amsart}

\usepackage{mathrsfs}
\usepackage{mathrsfs}
\usepackage{amsfonts}
\usepackage{amsfonts}
\usepackage{mathrsfs}
\usepackage{amsfonts}
\usepackage{amsfonts}
\usepackage{amsfonts}
\usepackage{mathrsfs}
\usepackage{amsfonts}
\usepackage{amssymb,amsmath}
\usepackage{amsthm}
\usepackage{amsfonts}

\newtheorem{theorem}{Theorem}[section]
\newtheorem{lemma}[theorem]{Lemma}

\theoremstyle{definition}

\theoremstyle{remark}
\newtheorem{remark}[theorem]{Remark}

\numberwithin{equation}{section}

\begin{document}

\title[Well-posedness of Boussinesq equations]
 {On the global well-posedness of a generalized 2D Boussinesq equations}

\author[J.X.Jia]{Junxiong Jia}
\address{Department of Mathematics,
Xi'an Jiaotong University,
 Xi'an
710049, China;}
\email{jiajunxiong@163.com}
\thanks{This work was supported by the National Science Foundation of China under contracts No. 11131006 and 60970149.}

\author[J.G. Peng]{Jigen Peng}
\address{Department of Mathematics,
Xi'an Jiaotong University,
 Xi'an
710049, China;}
\email{jgpeng@mail.xjtu.edu.cn}

\author[K.X. Li]{Kexue Li}
\address{Department of Mathematics,
Xi'an Jiaotong University,
 Xi'an
710049, China;}
\email{kexueli@gmail.com}






\begin{abstract}
In this paper, we consider the global solutions to a generalized 2D Boussinesq equation
\begin{align*}
\left \{\begin{aligned}
& \partial_{t} \omega + u\cdot \nabla \omega + \nu \Lambda^{\alpha} \omega = \theta_{x_{1}} ,  \quad \\
& u = \nabla^{\bot} \psi = (-\partial_{x_{2}} , \partial_{x_{1}}) \psi , \quad \Delta \psi = \Lambda^{\sigma} (\log (I-\Delta))^{\gamma} \omega ,  \quad \\
& \partial_{t} \theta + u\cdot \nabla \theta + \kappa \Lambda^{\beta} \theta = 0, \quad \\
& \omega(x,0) = \omega_{0}(x) , \quad \theta(x,0) = \theta_{0}(x),
\end{aligned}\right.
\end{align*}
with $\sigma \geq 0$, $\gamma \geq 0$, $\nu >0$, $\kappa>0$, $\alpha < 1$ and $\beta < 1$. When $\sigma = 0$,
$\gamma \geq 0$, $\alpha \in [0.95,1)$ and $\beta \in (1-\alpha,g(\alpha))$,
where $g(\alpha)<1$ is an explicit function as a technical bound, we prove that the above
equation has a global and unique solution in suitable functional space.
\end{abstract}

\maketitle


\section{Introduction}

The aim of this paper is prove that the following generalized 2D Boussinesq equation has a global solution in suitable functional settings.
\begin{align}
\label{GB}
\left \{\begin{aligned}
& \partial_{t} \omega + u\cdot \nabla \omega + \nu \Lambda^{\alpha} \omega = \theta_{x_{1}} ,  \quad \\
& u = \nabla^{\bot} \psi = (-\partial_{x_{2}} , \partial_{x_{1}}) \psi , \quad \Delta \psi = \Lambda^{\sigma} (\log (I-\Delta))^{\gamma} \omega ,  \quad \\
& \partial_{t} \theta + u\cdot \nabla \theta + \kappa \Lambda^{\beta} \theta = 0, \quad \\
& \omega(x,0) = \omega_{0}(x) , \quad \theta(x,0) = \theta_{0}(x),
\end{aligned}\right.
\end{align}
where $\omega = \omega (x,t)$, $\psi = \psi (x,t)$ and $\theta = \theta (x,t)$ are scalar functions of $x \in \mathbb{R}^{2}$
and $t \geq 0$, $u = u(x,t) \, : \, \mathbb{R}^{2} \rightarrow \mathbb{R}^{2} $ is a vector field, $0<\alpha<1$,$0<\beta<1$,$\sigma \geq 0$ and $\gamma \geq 0$ are real parameters, and $\Lambda = (-\Delta)^{\frac{1}{2}}$ and $\Lambda^{\sigma}$ are Fourier
multiplier operators with
\begin{align*}
\widehat{\Lambda^{\sigma}} f(\xi) = |\xi|^{\sigma} \widehat{f}(\xi).
\end{align*}
This generalized 2D Boussinesq equation proposed in \cite{T2BEWLSV} firstly.
In \cite{T2BEWLSV}, D. Chae and J. Wu proved that the above vorticity equation does have the velocity formation as follows
\begin{align}\label{velocity}
\left \{\begin{aligned}
& \partial_{t} v + u^{\bot}(\nabla^{\bot}\cdot v) + \nu \Lambda^{\alpha} v = -\nabla p + \theta e_{2}, \quad \\
& \nabla \cdot v = 0 , \quad u = \Lambda^{\sigma} (\log (I-\Delta))^{\gamma} v, \quad \\
& \partial_{t} \theta + u\cdot \nabla \theta + \kappa \Lambda^{\beta} \theta = 0, \quad \\
& v(x,0) = v_{0}(x), \quad \theta(x,t) = \theta_{0}(x).
\end{aligned}\right.
\end{align}
Obviously, the above model can be seen as a generalization of the 2D Boussinesq equations.

From physical view, Boussinesq type equations model the oceanic and atmospheric motions \cite{GFD}.
From the mathematical view, the fully viscos model with $\nu > 0$, $\kappa >0$, $\alpha = \beta =2$ is the simplest one to study.
And the most difficult one for the mathematical study is the inviscid model, that is when $\nu = \kappa = 0$.
In addition, the 2D Boussinesq equations acts very similar to the 3D Euler and Navier-Stokes equations, so it is hoped that the study of the 2D Boussinesq equations may shed light on the global regularity problem concerning the 3D Euler and Navier-Stokes equations.

Now, there are numerous studies about 2D Boussinesq equations. In 2006, D. Chae proved the global in time regularity for the
2D Boussinesq system with either the zero diffusivity or the zero viscosity in \cite{GRFT2BEWPVT}. In 2010, further progress has been
made by Hmidi et al., who proved the global regularity when the full Laplacian dissipation is replaced by the critical dissipation represented in terms of $\sqrt{-\Delta}$ \cite{GWPFABNSSWCD,GWPFEBSWCD}. Recently, Miao and Xue generalized the results to
accommodate both fractional dissipation and fractional thermal diffusion \cite{OTGWPOACOBNSS}. Some results
about 3D case have been obtained in \cite{GWPFTEBSWAD} by Hmidi et at. At the same time, some
other generalized models have been considered. M. Lai, R. Pan and K. Zhao studied the initial boundary value problem of 2D Boussinesq
equations over a bounded domain with smooth boundary \cite{IBVPFTDVBE}.
In 2011, C. Wang and Z. Zhang discussed the global well-posedness
for the 2D Boussinesq system with the temperature-dependent viscosity and thermal diffusivity \cite{GWPFT2BSWTTDVATD}.
In 2012, G. Wu and L. Xue showed that there is a global unique solution to the two-dimensional inviscid B$\acute{e}$nard system with fractional diffusivity system \cite{GWPFT2IBSWFDAYTD}.

We note that D. Chae and J. Wu only studied system (\ref{GB}) in the case $\nu >0$, $\kappa =0 $, $\alpha =0$ \cite{T2BEWLSV}.
Concerning the other cases, can we get similar results as in 2D Boussinesq system for system (\ref{GB}) which is a meaningful generalization. In this paper, we focus on the case $\nu >0$, $\kappa >0$, $0<\alpha<1$ and $0<\beta<1$.
Obviously, $\alpha$ and $\beta$ should satisfy the relation $\alpha + \beta \geq 1$, for the maximal gain of $\alpha$ derivative from the dissipation term should at least  roughly compensate the loss of one derivative in $\theta$ in the vorticity equation of system (\ref{GB}) with the help of the diffusion effect in the temperature equation. For brevity, we always set $\nu = \kappa = 1$
in the following. We shall adopt the subtle method used in \cite{T2BEWLSV,GWPFABNSSWCD,GWPFEBSWCD,OTGWPOACOBNSS} to study the coupled
effects of the generalized system. More precisely, we have the following result.

\begin{theorem}\label{global result}

Consider the generalized Boussinesq equations (\ref{GB}) with $\sigma = 0$ and $\gamma \geq 0$.
Let $\alpha \in [\frac{19}{20},1)$, $\beta \in (1- \alpha, g(\alpha))$ with
$g(\alpha) := \min \{ 2-2\alpha, \frac{8}{3} \, \alpha -2, $ $\frac{5\alpha (1-\alpha)}{11-10\alpha} \}$.
Assume the initial data $(\omega_{0},\theta_{0})$ satisfies $\omega_{0} \in L^{2}\cap L^{p}$ for any
$p > 2$ and $\theta_{0} \in H^{1-\alpha} \cap B_{\infty,1}^{1-\alpha+\epsilon}$
for arbitrary small $\epsilon >0$. Then (\ref{GB}) has a unique
global solution $(\omega,\theta)$ satisfying for any $t > 0$,
\begin{align*}
& \omega \in L_{t}^{\infty}L^{2} \cap L_{t}^{\infty}L^{p} \cap L_{t}^{1}B_{\infty,1}^{0,\gamma} \\
& \theta \in L_{t}^{\infty}(H^{1-\alpha} \cap B_{\infty,1}^{1-\alpha+\epsilon}) \cap
L_{t}^{1}(H^{1-\alpha+\beta} \cap B_{\infty,1}^{1-\alpha+\beta+\epsilon}).
\end{align*}
\end{theorem}

For the definitions of Besov space $B_{p,q}^{s}$, generalized Besov space $B_{p,q}^{s,\gamma}$ and mixed space-time Besov space see
the next section below. Now, we should give some comments.

\begin{remark}
We know that the case for $\alpha < 1$, $\beta < 1$ and $\alpha + \beta \geq 1$ is nontrivial.
Until now there is no effective way to treat this case, for the regularization from the fractional diffusion
term  not strong enough. So we have to exploit the structure of the system to overcome the difficulty.
In this paper, the method is workable but very restrictive.
\end{remark}

\begin{remark}
In our theorem, we need $\beta$ smaller than a very complicated function.
It is a technical assumption.
In common sense, the term $\Lambda^{\beta}$ is a good term when $\beta$ is large.
So we can gauss that
the result in Theorem \ref{global result} is hold
 for $\alpha \in [\frac{19}{20},1)$, $\beta \in (1-\alpha,1)$.
\end{remark}

To prove Theorem \ref{global result}, there are two main difficulties. Firstly, following the procedure as in \cite{T2BEWLSV},
we will encounter the operator like $\mathcal{R}_{\alpha} = \Lambda^{-\alpha} \partial_{1}$ which is different from Riesz Transform and is not a bounded operator in $L^{p}$ space. So the technique used in \cite{T2BEWLSV} can not be used here without significant changes.
On the other hand, considering the structure of the system (\ref{velocity}), we hardily obtain the $L^{2}$ estimates of $v$.
Hence, the techniques used in \cite{OTGWPOACOBNSS} also need lots of nontrivial changes.

The paper is organized as follows. In Section 2, we list some useful results about Besov space and some estimates which will be used in our
proof. Section 3 is devoted to prove some priori estimates which are the main part of this paper. In Section 4, we give the proof of the uniqueness part of Theorem \ref{global result}. Finally, some technical lemmas are shown in Section 5.

\section{Preliminaries}

Throughout this paper we will use the following notations.
\begin{itemize}
  \item For any tempered distribution $u$ both $\widehat{u}$ and $\mathcal{F}u$ denote the Fourier transform of $u$.
  \item For every $p\in [1,\infty]$, $\|\cdot\|_{L^{p}}$ denotes the norm in the Lebesgue space $L^{p}$.
  \item The norm in the mixed space-time Lebesgue space $L^{p}([0,T];L^{r}(\mathbb{R}^{d}))$ is denoted by $\|\cdot\|_{L^{p}_{T}L^{r}}$ (with the obvious generalization to $\|\cdot\|_{L^{p}_{T}X}$ for any normed space X).
  \item For any pair of operators $P$ and $Q$ on some Banach space $X$, the commutator $[P,Q]$ is given by $PQ-QP$.
\end{itemize}

Then, we give a short introduction to the Besov type space. Details about Besov type space can be found in \cite{FANPDE} or \cite{PIF}.
There exist two radial positive functions $\chi \in \mathcal{D}(\mathbb{R}^{d})$ and $\varphi \in \mathcal{D}(\mathbb{R}^{d}\backslash \{0\})$
such that
\begin{itemize}
  \item $\chi (\xi) + \sum_{q\geq 0} \varphi (2^{-q} \xi) =1$; $\forall q \geq 1$,
        $\text{supp} \chi \cap \text{supp} \varphi (2^{-q}\cdot) = \phi$,
  \item $\text{supp} \varphi (2^{-j}\cdot) \cap \text{supp} \varphi (2^{-k}\cdot) = \phi$, if $|j-k|\geq 2$,
\end{itemize}

For every $v \in S^{'}(\mathbb{R}^{d})$ we set
\begin{align*}
\Delta_{-1}v = \chi(D)v, \quad \forall q \in \mathbb{N}, \quad \Delta_{j}v = \varphi (2^{-q}D)v \quad \text{and} \quad
S_{j} = \sum_{-1\leq m\leq j-1} \Delta_{m}.
\end{align*}
The homogeneous operators are defined by
\begin{align*}
\dot{\Delta}_{q}v = \varphi (2^{-q}D)v, \quad \dot{S}_{j}=\sum_{m\leq j-1} \dot{\Delta}_{j}v, \quad \forall q \in \mathbb{Z}.
\end{align*}
One can easily verifies that with our choice of $\varphi$,
\begin{align}\label{dd2}
\Delta_{j}\Delta_{k}f=0\quad \text{if} \quad |j-k|\geq 2
\end{align}
\begin{align}\label{ds4}
\Delta_{j}(S_{k-1}f \Delta_{k}f)=0\quad \text{if} \quad |j-k|\geq 5.
\end{align}
As in Bony's decomposition, we split the product $uv$ into three parts
\begin{align*}
uv=T_{u}v + T_{v}u +R(u,v),
\end{align*}
with
\begin{align*}
T_{u}v=\sum_{j}S_{j-1}u \Delta_{j}v,
\end{align*}
\begin{align*}
R(u,v)=\sum_{j} \Delta_{j}u \widetilde{\Delta}_{j}v
\end{align*}
where $\widetilde{\Delta}_{j} = \Delta_{j-1} + \Delta_{j} + \Delta_{j+1}$.

Let us now define inhomogeneous Besov spaces. For $(p,q)\in [1, +\infty]^{2}$ and $s\in \mathbb{R}$ we define the inhomogeneous
Besov space $B_{p,q}^{s}$ as the set of tempered distributions $u$ such that
\begin{align*}
\|u\|_{B_{p,q}^{s}}:=(2^{js}\|\Delta_{j}u\|_{L^{p}})_{\ell^{q}} < +\infty.
\end{align*}
The homogeneous Besov space $\dot{B}_{p,q}^{s}$ is defined as the set of $u \in S^{'}(\mathbb{R}^{d})$ up to polynomials such that
\begin{align*}
\|u\|_{\dot{B}_{p,q}^{s}}:=(2^{js}\|\dot{\Delta}_{j}u\|_{L^{p}})_{\ell^{q}} < +\infty.
\end{align*}
Notice that the usual Sobolev spaces $H^{s}$ coincide with $B_{2,2}^{s}$ for every $s \in \mathbb{R}$ and that the homogeneous
spaces $\dot{H}^{s}$ coincide with $\dot{B}_{2,2}^{s}$.

For $s \in (0,1)$ and $1 \leq p,q \leq \infty$, we can define Besov spaces equivalently as follows
\begin{align}
& \|u\|_{\dot{B}_{p,q}^{s}} = \left( \int_{\mathbb{R}^{d}} \frac{(\|u(x+t)-u(x)\|_{L^{p}})^{q}}{|t|^{d+sq}} \, dt \right)^{1/q}, \\
& \|u\|_{B_{p,q}^{s}} = \|u\|_{L^{p}} + \left( \int_{\mathbb{R}^{d}} \frac{(\|u(x+t)-u(x)\|_{L^{p}})^{q}}{|t|^{d+sq}} \, dt \right)^{1/q}.
\end{align}
When $q=\infty$, the expressions are interpreted in the normal way.

We shall need some mixed space-time spaces. Let $T>0$ and $\rho \geq 1$, we denote by $L^{\rho}_{T}B_{p,q}^{s}$ the space of
distribution $u$ such that
\begin{align*}
\|u\|_{L^{\rho}_{T}\dot{B}_{p,q}^{s}}:=\|(2^{js}\|\dot{\Delta}_{j}u\|_{L^{p}})_{\ell^{q}}\|_{L^{\rho}_{T}} <+\infty.
\end{align*}
We say that $u$ belongs to the space $\widetilde{L}_{T}^{\rho}B_{p,q}^{s}$ if
\begin{align*}
\|u\|_{\widetilde{L}^{\rho}_{T}\dot{B}_{p,q}^{s}}:=(2^{js}\|\dot{\Delta}_{j}u\|_{L^{\rho}_{T}L^{p}})_{\ell^{q}} <+\infty.
\end{align*}
Through a direct application of the Minkowski inequality, the following links between these spaces is true \cite{T.Hmidi2010JDE}.
Let $\varepsilon >0$, then
\begin{align*}
L^{\rho}_{T}B_{p,q}^{s} \hookrightarrow \widetilde{L}^{\rho}_{T}B_{p,q}^{s} \hookrightarrow L^{\rho}_{T}B_{p,q}^{s-\varepsilon},
\quad \text{if} \,\, q\geq \rho,
\end{align*}
\begin{align*}
L^{\rho}_{T}B_{p,q}^{s+\varepsilon} \hookrightarrow \widetilde{L}^{\rho}_{T}B_{p,q}^{s} \hookrightarrow
L^{\rho}_{T}B_{p,q}^{s}, \quad \text{if} \,\, \rho \geq q.
\end{align*}

Then, we give the definition of a generalized Besov spaces which include an algebraic part of the modes.
For $s,\gamma \in \mathbb{R}$ and $1 \leq p,q \leq \infty$, the generalized Besov spaces $\dot{B}_{p,q}^{s,\gamma}$ and
$B_{p,q}^{s,\gamma}$ are defined by
\begin{align*}
& \|u\|_{\dot{B}_{p,q}^{s,\gamma}} := \|2^{js}(1+|j|)^{\gamma}\|\dot{\Delta}_{j}u\|_{L^{p}}\|_{\ell^{q}} < \infty, \\
& \|u\|_{\dot{B}_{p,q}^{s,\gamma}} := \|2^{js}(1+|j|)^{\gamma}\|\Delta_{j}u\|_{L^{p}}\|_{\ell^{q}} < \infty.
\end{align*}
The space $L_{t}^{\rho}B_{p,q}^{s,\gamma}$, $L_{t}^{\rho}\dot{B}_{p,q}^{s,\gamma}$, $\widetilde{L}_{t}^{\rho}B_{p,q}^{s,\gamma}$ and
$\widetilde{L}_{t}^{\rho}\dot{B}_{p,q}^{s,\gamma}$ are defined similar as in Besov space.

Bernstein type inequalities for fractional derivatives and Osgood inequality are often used in our proof.
For reader's convenience, we list them here.

\begin{lemma}\label{bernstein}
Let $\alpha \geq 0$. Let $1 \leq p \leq q \leq \infty$.

(1)If $u$ satisfies
\begin{align*}
\sup \hat{u} \subset \left\{ \xi \in \mathbb{R}^{d} \, : \, |\xi|\leq K2^{j} \right\},
\end{align*}
for some integer $j$ and a constant $K > 0$, then
\begin{align*}
\|(-\Delta)^{\alpha}u\|_{L^{q}(\mathbb{R}^{d})} \leq C_{1} 2^{2\alpha j + jd(\frac{1}{p} - \frac{1}{q})} \|u\|_{L^{p}(\mathbb{R}^{d})}.
\end{align*}

(2)If $u$ satisfies
\begin{align*}
\sup \widehat{u} \subset \left\{ \xi \in \mathbb{R}^{d} \, : \, K_{1}2^{j} \leq |\xi| \leq K_{2}2^{j} \right\}
\end{align*}
for some integer $j$ and constants $0 < K_{1} \leq K_{2}$, then
\begin{align*}
C_{1} 2^{2 \alpha j}\|u\|_{L^{q}(\mathbb{R}^{d})} \leq \|(-\Delta)^{\alpha}u\|_{L^{q}(\mathbb{R}^{d})} \leq
C_{2}2^{2\alpha j+jd(\frac{1}{p}-\frac{1}{q})}\|u\|_{L^{p}(\mathbb{R}^{d})}
\end{align*}
where $C_{1}$ and $C_{2}$ are constants depending on $\alpha$, $p$ and $q$ only.
\end{lemma}
\begin{lemma}\label{osgood}
Let $\alpha(t)>0$ be a locally integrable function. Assume $\omega (t)\geq 0$ satisfies
\begin{align*}
\int_{0}^{\infty}\frac{1}{\omega (r)} \, dr = \infty .
\end{align*}
Suppose that $\rho (t)>0$ satisfies
\begin{align*}
\rho (t) \leq a + \int_{t_{0}}^{t} \alpha(s)\omega (\rho (s)) \, ds
\end{align*}
for some constant $a\geq 0$. Then if $a=0$, then $\rho = 0$; if $a>0$, then
\begin{align*}
-\Omega(\rho(t))+\Omega(a)\leq \int_{t_{0}}^{t}\alpha(\tau) \, d\tau,
\end{align*}
where
\begin{align*}
\Omega(x) = \int_{x}^{1} \frac{1}{\omega(r)} \, dr.
\end{align*}
\end{lemma}

At the end of this section, we collect some useful estimates for the smooth solutions of the following linear transport-diffusion equation
\begin{align}\label{linear transport}
\left \{\begin{aligned}
& \partial_{t} \theta + u \cdot \nabla \theta + \Lambda^{\beta}\theta = f, \quad \beta \in [0,1] \\
& \text{div} u = 0, \quad \theta(x,0)=\theta_{0}(x).
\end{aligned}\right.
\end{align}

The following Lemmas can be found in \cite{OTGWPOACOBNSS,AMPATTQGW,OTGWPOTTDBSWAZV,GWPOTCBEICBS}.
\begin{lemma}\label{Lpestimate}
Let $u$ be a smooth divergence-free vector field of $\mathbb{R}^{d}$ and $\theta$ be a smooth solution of equation (\ref{linear transport}).
Then for every $p \in [1,\infty]$ we have
\begin{align*}
\|\theta(t)\|_{L^{p}}\leq \|\theta_{0}\|_{L^{p}} + \int_{0}^{t} \|f(\tau)\|_{L^{p}}\, d\tau.
\end{align*}
\end{lemma}
\begin{lemma}\label{est}
Let $u$ be a smooth divergence-free vector field of $\mathbb{R}^{d}$ with vorticity $\omega$ be a smooth solution of equation (\ref{linear transport}).
Then for every $(p,\rho) \in (1,\infty) \times [1,\infty]$, we have
\begin{align*}
\sup_{j \in \mathbb{N}} 2^{j\frac{\beta}{\rho}} \|\Delta_{j} u\|_{L^{\rho}_{t}L^{p}}
\lesssim_{\rho, p} \|\theta_{0}\|_{L^{p}} + \|\theta_{0}\|_{L^{\infty}} \|\omega\|_{L_{t}^{1}L^{p}} + \|f\|_{L_{t}^{1}L^{p}}.
\end{align*}
\end{lemma}
\begin{lemma}\label{expo}
Let $-1 < s < 1$, $(\rho, \rho_{1}, p, r) \in [1,\infty]^{4}$, $\rho_{1}\leq \rho$ and $u$ be a divergence-free vector field belonging
to $L_{\text{loc}}^{1}(\mathbb{R}^{+}; \text{Lip}(\mathbb{R}^{d}))$. We consider a smooth solution $\theta$ of the equation (\ref{linear transport}), then there exists $C > 0$ such that for every $t \in \mathbb{R}^{+}$,
\begin{align*}
\|\theta\|_{\widetilde{L}_{t}^{\infty}B_{p,q}^{s}} + \|(\text{Id} - \Delta_{-1})\theta\|_{\widetilde{L}_{t}^{\rho}B_{p,q}^{s+\frac{\beta}{\rho}}}
\leq C e^{C U(t)} \left( \|\theta_{0}\|_{B_{p,q}^{s}} + \|f\|_{L_{t}^{1}B_{p,q}^{s}} \right),
\end{align*}
and
\begin{align*}
\|\theta\|_{\widetilde{L}_{t}^{\rho}\dot{B}_{p,q}^{s+\frac{\beta}{\rho}}} \leq C e^{C U(t)}
\left( \|\theta_{0}\|_{\dot{B}_{p,q}^{s}} + \|f\|_{\widetilde{L}_{t}^{\rho_{1}}\dot{B}_{p,q}^{s+\frac{\beta}{\rho_{1}}-\beta}} \right),
\end{align*}
where $U(t) := \int_{0}^{t} \|\nabla u(\tau)\|_{L^{\infty}} \, d\tau$.
\end{lemma}

\section{Commutator estimates}

First we recall a pseudo-differential operator $\mathcal{R}_{\alpha}$
defined by $\mathcal{R}_{\alpha} := \Lambda^{-\alpha}\partial_{1}$, $0 < \alpha < 1$.
Considering $\mathcal{R}_{\alpha} = \Lambda^{1-\alpha} \mathcal{R}$, where $\mathcal{R}$ is the usual Riesz transform,
we call $\mathcal{R}_{\alpha}$ the modified Riesz transform. The following theorem can be found in \cite{OTGWPOACOBNSS}.
\begin{theorem}\label{generalized riesz}
Let $0 < \alpha < 1$, $q \in \mathbb{N}$ and $\mathcal{R}_{\alpha} := \frac{\partial_{1}}{\Lambda^{\alpha}}$ be the
modified Riesz transform.

(1)Let $\chi \in \mathcal{D}(\mathbb{R}^{d})$. Then for every $(p,s) \in [1,\infty] \times (\alpha -1,\infty)$,
\begin{align*}
\|\Lambda^{s}\chi(2^{-q}\Lambda)\mathcal{R}_{\alpha}\|_{\mathcal{L}(L^{p})} \lesssim 2^{q(s+1-\alpha)}.
\end{align*}

(2)Let $\mathcal{C}$ be a ring. Then there exists $\phi \in \mathcal{S}(\mathbb{R}^{d})$ whose spectrum does not meet
the origin such that
\begin{align*}
\mathcal{R}_{\alpha}u = 2^{q(d+1-\alpha)}\phi (2^{q} \cdot)*u
\end{align*}
for every $u$ with Fourier variable supported on $2^{q}\mathcal{C}$.
\end{theorem}

The following Lemma is useful in dealing with the commutator terms.
\begin{lemma}\label{commutator}
Let $p \in [1,\infty]$ and $\delta \in (0,1)$.

(1)If $|x|^{\delta}\phi \in L^{1}$, $f \in \dot{B}_{p,\infty}^{\delta}$ and $g \in L^{\infty}$, then
\begin{align}
\|\phi * (fg) - f(\phi *g)\|_{L^{p}} \leq C \||x|^{\delta}\phi\|_{L^{1}} \|f\|_{\dot{B}_{p,\infty}^{\delta}} \|g\|_{L^{\infty}}.
\end{align}
In the case when $\delta = 1$, we have
\begin{align}
\|\phi *(fg) - f (\phi *g)\|_{L^{p}} \leq C \||x|\phi\|_{L^{1}} \|\nabla f\|_{L^{p}} \|g\|_{L^{\infty}}.
\end{align}

(2)If $|x|^{\delta}\phi \in L^{1}$, $f \in \dot{B}_{\infty,\infty}^{\delta}$ and $g \in L^{p}$, then
\begin{align}
\|\phi * (fg) - f(\phi *g)\|_{L^{p}} \leq C \||x|^{\delta}\phi\|_{L^{1}} \|f\|_{\dot{B}_{\infty,\infty}^{\delta}} \|g\|_{L^{p}}.
\end{align}
In the case when $\delta = 1$, we have
\begin{align}
\|\phi *(fg) - f (\phi *g)\|_{L^{p}} \leq C \||x|\phi\|_{L^{1}} \|\nabla f\|_{L^{\infty}} \|g\|_{L^{p}}.
\end{align}
\end{lemma}

The first part is proved in \cite{T2BEWLSV}, so we just prove the second part of the above Lemma.

\begin{proof}
By Minkowski's inequality, for any $p \in [1,\infty]$,
\begin{align*}
&\quad\,\, \|\phi *(fg) - f(\phi *g)\|_{L^{p}} \\
& = \left[ \int \left| \int \phi(z)\left(f(x)-f(x-z)\right)g(x-z)\, dz\right|^{p} \, dx\right]^{1/p} \\
& \leq \int \left[ \int \left| \phi(z)\left( f(x)-f(x-z) \right) g(x-z) \right|^{p} \, dx \right]^{1/p} \, dz \\
& \leq \int \frac{\|f(\cdot)-f(\cdot -z)\|_{L^{\infty}}}{|z|^{\delta}} |z|^{\delta} |\phi(z)| \left( \int |g(x-z)|^{p} \, dx\right)^{1/p} \, dz   \\
& \leq \||x|^{\delta}\phi\|_{L^{1}} \sup_{|z|>0} \frac{\|f(\cdot)-f(\cdot -z)\|_{L^{\infty}}}{|z|^{\delta}} \|g\|_{L^{p}} \\
& \leq C \||x|^{\delta}\phi\|_{L^{1}}\|f\|_{\dot{B}_{\infty,\infty}^{\delta}} \|g\|_{L^{p}}.
\end{align*}
In the last inequality, we use the definition of $\dot{B}_{\infty,\infty}^{\delta}$.
\end{proof}

The next Theorem concerns the crucial commutators involving $\mathcal{R}_{\alpha}$.

\begin{theorem}\label{crucial commutators}
Let $\alpha \in (0,1)$,$\epsilon >0$ taken to be small enough, $u$ be a smooth divergence-free vector field of $\mathbb{R}^{d}$ and $\theta$ be a smooth scalar function. Then,

(1)For every $(s,p,q) \in (-1,\alpha - \sigma)\times [2,\infty]\times [1,\infty]$ and take $\epsilon >0$ satisfy
$s+\sigma +\epsilon < \alpha$ we have
\begin{align}
\|[\mathcal{R}_{\alpha},u\cdot \nabla]\theta\|_{B_{p,q}^{s}} \lesssim \|u \|_{\dot{B}_{p,\infty}^{1-\sigma-\epsilon}}
\left( \|\theta\|_{B_{\infty,q}^{s+1+\sigma+\epsilon-\alpha}} + \|\theta\|_{L^{\infty}} \right).
\end{align}

(2)In the 2 dimensional case, if $u= \nabla^{\perp} \Delta^{-1} \Lambda^{\sigma}\left(\log(\text{Id}-\Delta)\right)^{\gamma}\omega$
, $\omega := G + \mathcal{R}_{\alpha}\theta$ and $0\leq \sigma <\alpha <1$.
Then for every $0<s<\alpha -\sigma$, taking arbitrary small $\epsilon >0$ such that $s+\sigma+\epsilon <\alpha$,
we have
\begin{align}
\|[\mathcal{R}_{\alpha},u]\theta\|_{H^{s}} \lesssim &\|G\|_{L^{2}} \|\theta\|_{B_{\infty,2}^{s+\sigma+\epsilon-\alpha}}
+ \|\theta\|_{L^{\infty}} \|\theta\|_{H^{s+1+\sigma +\epsilon -2\alpha}}  \\
&  + \|G\|_{L^{2}} \|\theta\|_{L^{\frac{2}{1-\sigma}}}
+\|\theta\|_{L^{2}}\|\theta\|_{L^{\frac{2}{1-\sigma}}}.
\end{align}
\end{theorem}
\begin{proof}
(1)
Due to Bony's decomposition we split the commutator term into three parts
\begin{align*}
[\mathcal{R}_{\alpha},u\cdot \nabla]\theta  = &\sum_{n\in \mathbb{N}}[\mathcal{R}_{\alpha},S_{n-1}u\cdot \nabla]\Delta_{n}\theta
+ \sum_{n\in \mathbb{N}}[\mathcal{R}_{\alpha},\Delta_{n}u\cdot \nabla]S_{n-1}\theta    \\
& + \sum_{n\geq -1}[\mathcal{R}_{\alpha},\Delta_{n}u\cdot \nabla]\widetilde{\Delta}_{n}\theta   \\
= &I+II+III
\end{align*}
For I, since for every $n \in \mathbb{N}$ the Fourier transform of $S_{n-1}u\Delta_{n}\theta$ is supported in a ring of size $2^{n}$,
then from Theorem \ref{generalized riesz} and Lemma \ref{commutator}, we have for every $j\geq -1$
\begin{align*}
\|\Delta_{j}I\|_{L^{p}} & \lesssim \sum_{|n-j|\leq 4} \|[\phi_{n}*,S_{n-1}u\cdot \nabla]\Delta_{n}\theta\|_{L^{p}}  \\
& \lesssim \sum_{|n-j|\leq 4} 2^{n(\sigma + \epsilon - \alpha)} \|u\|_{\dot{B}_{p,\infty}^{1-\sigma-\epsilon}}2^{n}\|\Delta_{n}\theta\|_{L^{\infty}} \\
& \lesssim c_{j}2^{-js} \|u\|_{\dot{B}_{p,\infty}^{1-\sigma-\epsilon}} \|\theta\|_{B_{\infty,q}^{s+\sigma +\epsilon +1 -\alpha}},
\end{align*}
where $\phi_{n}(x) := 2^{n(d+1-\alpha)} \phi(2^{n}x)$ with $\phi \in \mathcal{S}$ and $(c_{j})_{j\geq -1}$
with $\|c_{j}\|_{\ell^{q}}=1$.
Thus we obtain
\begin{align*}
\|I\|_{B_{p,q}^{s}} \lesssim \|u\|_{\dot{B}_{p,\infty}^{1-\sigma-\epsilon}} \|\theta\|_{B_{\infty,q}^{s+\sigma+\epsilon +1-\alpha}}.
\end{align*}
For II, similar to I we have
\begin{align*}
\|\Delta_{j}II\|_{L^{p}} & \lesssim \sum_{|n-j|\leq 4,n \in \mathbb{N}} \|[\phi_{n}*,\Delta_{n}u\cdot \nabla]S_{n-1}\theta\|_{L^{p}}  \\
& \lesssim \sum_{|n-j|\leq 4}2^{n(\sigma+\epsilon-\alpha)}\|u\|_{\dot{B}_{p,\infty}^{1-\sigma-\epsilon}}\|\nabla S_{n-1}\theta\|_{L^{\infty}}    \\
& \lesssim \|u\|_{\dot{B}_{p,\infty}^{1-\sigma-\epsilon}}2^{-js}\sum_{-1\leq n^{'} \leq j+2} 2^{(n^{'}-j)(\alpha -\sigma -\epsilon -s)}
2^{n^{'}(s+\sigma +\epsilon -\alpha +1)}\|\Delta_{n^{'}}\theta\|_{L^{\infty}}.
\end{align*}
Thus using discrete Young's inequality we obtain
\begin{align*}
\|II\|_{B_{p,q}^{s}} \lesssim \|u\|_{\dot{B}_{p,\infty}^{1-\sigma-\epsilon}} \|\theta\|_{B_{\infty,q}^{s+\sigma+\epsilon +1-\alpha}}.
\end{align*}
For III, we further write
\begin{align*}
III & = \sum_{n \geq 0}\text{div} [\mathcal{R}_{\alpha},\Delta_{n}u]\widetilde{\Delta}_{n}\theta +
\sum_{1\leq i \leq n} [\partial_{i}\mathcal{R}_{\alpha},\Delta_{-1}u^{i}]\widetilde{\Delta}_{-1}\theta  \\
& = III^{1}+III^{2}.
\end{align*}
Considering Bernstein's inequality and Theorem \ref{generalized riesz}, we deal with the term $III^{1}$ as follows
\begin{align*}
& \|\Delta_{j}III^{1}\|_{L^{p}} \\
& \lesssim \sum_{n\geq j-3 , n\geq 0} \|\Delta_{j}\text{div} \mathcal{R}_{\alpha} (\Delta_{n}
u\widetilde{\Delta}_{n}\theta)\|_{L^{p}} + \sum_{n\geq j-3, n\geq 0} \|\Delta_{j}\text{div}(\Delta_{n}u \mathcal{R}_{\alpha}
\widetilde{\Delta}_{n}\theta)\|_{L^{p}} \\
& \lesssim \sum_{n\geq j-3}(2^{j(2-\alpha)}+2^{j}2^{n(1-\alpha)})2^{-n(1-\sigma-\epsilon)}\|\Delta_{n}\Lambda^{1-\sigma-\epsilon}u\|_{L^{p}}
\|\widetilde{\Delta}_{n}\theta\|_{L^{\infty}} \\
& \lesssim \|u\|_{\dot{B}_{p,\infty}^{1-\sigma-\epsilon}}2^{-js}\sum_{n\geq j-4}
(2^{(j-n)(s+2-\alpha)}+2^{(j-n)(s+1)})2^{n(s+1+\sigma +\epsilon -\alpha)}\|\Delta_{n}\theta\|_{L^{\infty}}.
\end{align*}
Thus we obtain for every $s > -1$
\begin{align*}
\|III^{1}\|_{B_{p,q}^{s}}\lesssim \|u\|_{\dot{B}_{p,\infty}^{1-\sigma-\epsilon}} \|\theta\|_{B_{\infty,q}^{s+1+\sigma+\epsilon -\alpha}}
\end{align*}
Choosing a suitable function $\chi \in \mathcal{D}(\mathbb{R}^{d})$, we have
\begin{align*}
III^{2}=\sum_{1\leq i\leq n}[\partial_{i}\mathcal{R}_{\alpha}\chi(D),\Delta_{-1}u^{i}]\widetilde{\Delta}_{-1}\theta.
\end{align*}
By Theorem \ref{generalized riesz}, we know that $\partial_{i}\mathcal{R}_{\alpha}\chi(D)$ is a convolution operator with
kernel $h$ satisfying
\begin{align*}
|h(x)|\leq C(1+|x|)^{-d-2+\alpha}, \quad \text{for all} \,\,\, x\in \mathbb{R}^{d}.
\end{align*}
Next, from the fact that $\Delta_{j}III^{2} = 0$ for every $j \geq 3$ and using Lemma \ref{commutator}, we have
\begin{align*}
\|III^{2}\|_{B_{p,q}^{s}} & \lesssim \|[h*,\Delta_{-1}u]\widetilde{\Delta}_{-1}\theta\|_{L^{p}} \\
& \lesssim \||x|^{1-\sigma-\epsilon}h\|_{L^{1}}\|\Delta_{-1}u\|_{\dot{B}_{p,\infty}^{1-\sigma-\epsilon}}\|\widetilde{\Delta}_{-1}\theta\|_{L^{\infty}} \\
& \lesssim \|u\|_{\dot{B}_{p,\infty}^{1-\sigma-\epsilon}} \|\theta\|_{L^{\infty}}.
\end{align*}
Hence, the proof of the first part is complete.

(2)
As in the first part, we can get the following equality by Bony's decomposition.
\begin{align*}
[\mathcal{R}_{\alpha},u]\theta & = \sum_{n\in \mathbb{N}} [\mathcal{R}_{\alpha},S_{n-1}u]\Delta_{n}\theta
+\sum_{n\in \mathbb{N}} [\mathcal{R}_{\alpha},\Delta_{n}u]S_{n-1}\theta + \sum_{n\geq -1}[\mathcal{R}_{\alpha},\Delta_{n}u]\widetilde{\Delta}_{n}\theta  \\
& = I + II + III.
\end{align*}
For brevity, let $P(\Lambda):=\Lambda^{\sigma}\left( \log(\text{Id}-\Delta) \right)^{\gamma}$.
For I, denote $I_{n} := [\mathcal{R}_{\alpha},S_{n-1}u]\Delta_{n}\theta$. Since for each $n\in \mathbb{N}$
the Fourier transform of $S_{n-1}u\Delta_{n}\theta$ is supported in a ring of size $2^{n}$, from Theorem \ref{generalized riesz} there
exists $\phi \in \mathcal{S}(\mathbb{R}^{2})$ whose spectrum is away from the origin such that
\begin{align*}
I_{n} = [\phi_{n}*,S_{n-1}\nabla^{\bot}\Delta^{-1}P(\Lambda)G]\Delta_{n}\theta + [\phi_{n}*,S_{n-1}\nabla^{\bot}\Delta^{-1}
P(\Lambda)\mathcal{R}_{\alpha}\theta]\Delta_{n}\theta,
\end{align*}
where $\phi_{n}(x):=2^{n(d+1-\alpha)}\phi(2^{n}x)$. Using Lemma \ref{commutator} and Theorem \ref{generalized riesz}, we obtain
\begin{align*}
\|I_{n}\|_{L^{2}} \lesssim & \||x|^{1-\sigma-\epsilon}\phi_{n}\|_{L^{1}} \|\Lambda^{1-\sigma-\epsilon}S_{n-1}
\nabla^{\bot}\Delta^{-1}P(\Lambda)G\|_{L^{2}} \|\Delta_{n}\theta\|_{L^{\infty}} \\
 & + \||x|^{1-\sigma-\epsilon}\phi_{n}\|_{L^{1}} \|\Lambda^{1-\sigma-\epsilon}S_{n-1}\nabla^{\bot}\Delta^{-1}P(\Lambda)
 \mathcal{R}_{\alpha}\theta\|_{L^{\infty}} \|\Delta_{n}\theta\|_{L^{2}}  \\
\lesssim & 2^{n(\sigma+\epsilon -\alpha)}\||x|^{1-\sigma-\epsilon}\phi\|_{L^{1}} \|G\|_{L^{2}} \|\Delta_{n}\theta\|_{L^{\infty}} \\
& + 2^{n(\sigma + \epsilon -\alpha)}2^{n(1-\alpha)} \||x|^{1-\sigma -\epsilon }\phi\|_{L^{1}} \|\theta\|_{L^{\infty}}\|\Delta_{n}\theta\|_{L^{2}} \\
\lesssim & 2^{n(\sigma +\epsilon -\alpha)}\|G\|_{L^{2}}\|\Delta_{n}\theta\|_{L^{\infty}}
+ 2^{n(\sigma +\epsilon -\alpha)}2^{n(1-\alpha)}\|\theta\|_{L^{\infty}}\|\Delta_{n}\theta\|_{L^{2}}.
\end{align*}
Thus, we can obtain
\begin{align*}
\|I\|_{H^{s}} \lesssim \|G\|_{L^{2}} \|\theta\|_{B_{\infty,2}^{s+\sigma+\epsilon -\alpha}}
+ \|\theta\|_{L^{\infty}} \|\theta\|_{H^{s+\sigma +\epsilon +1- 2\alpha}}.
\end{align*}
For II, denote $II_{n}:=[\mathcal{R}_{\alpha},\Delta_{n}u]S_{n-1}\theta$. As in I, we have
\begin{align*}
II_{n} = [\phi_{n}*,\Delta_{n}\Delta^{-1}\nabla^{\bot}P(\Lambda)G]S_{n-1}\theta + [\phi_{n}*,\Delta_{n}\Delta^{-1}\nabla^{\bot}
P(\Lambda)\mathcal{R}_{\alpha}\theta]S_{n-1}\theta
\end{align*}
By Lemma \ref{commutator}, we get
\begin{align*}
\|II_{n}\|_{L^{2}} \lesssim & \||x|^{1-\sigma-\epsilon}\phi_{n}\|_{L^{1}} \|\Lambda^{1-\sigma-\epsilon}\Delta_{n}\nabla^{\bot}\Delta^{-1}P(\Lambda)G\|_{L^{2}} \|S_{n-1}\theta\|_{L^{\infty}} \\
& + \||x|^{1-\sigma-\epsilon}\phi_{n}\|_{L^{1}} \|\Lambda^{1-\sigma-\epsilon}\Delta_{n}\nabla^{\bot}\Delta^{-1}P(\Lambda)\mathcal{R}_{\alpha}\theta\|_{L^{2}}\|S_{n-1}\theta\|_{L^{\infty}} \\
& \lesssim 2^{n(\sigma + \epsilon -\alpha)} \|G\|_{L^{2}} \|S_{n-1}\theta\|_{L^{\infty}}
+ 2^{n(\sigma + \epsilon - \alpha)} 2^{n(1-\alpha)} \|\Delta_{n}\theta\|_{L^{2}}\|\theta\|_{L^{\infty}}.
\end{align*}
Hence, by discrete Young's inequality, we have
\begin{align*}
\|II\|_{H^{s}} \lesssim \|G\|_{L^{2}} \|\theta\|_{B_{\infty,2}^{s+\sigma +\epsilon -\alpha}} +
\|\theta\|_{L^{\infty}} \|\theta\|_{H^{s+1+\sigma +\epsilon -2 \alpha}}.
\end{align*}
We further write III as follows
\begin{align*}
III & = \sum_{n\geq 0} \mathcal{R}_{\alpha}(\Delta_{n}u\widetilde{\Delta}_{n}\theta) + \sum_{n\geq 0}\Delta_{n}u
\mathcal{R}_{\alpha} \widetilde{\Delta}_{n}\theta + [\mathcal{R}_{\alpha},\Delta_{-1}u]\widetilde{\Delta}_{-1}\theta    \\
& = III^{1} + III^{2} + III^{3}.
\end{align*}
First, we note that for every $n \geq 0$
\begin{align*}
\|\Delta_{n}u\|_{L^{2}} \lesssim 2^{n(\sigma+\epsilon-1)}\|\Delta_{n}G\|_{L^{2}} + 2^{n(\sigma+\epsilon-\alpha)} \|\Delta_{n}\theta\|_{L^{2}}.
\end{align*}
Then by a direct computation we have
\begin{align*}
& 2^{js}\|\Delta_{j}III^{1}\|_{L^{2}} \\
\lesssim & 2^{j(s+1-\alpha)} \sum_{n\geq j-3, n\geq 0} \|\Delta_{n}u\|_{L^{2}}\|\widetilde{\Delta}_{n}\theta\|_{L^{\infty}} \\
\lesssim & 2^{j(s+1-\alpha)}\sum_{n\geq j-3}\left( 2^{n(\sigma+\epsilon-1)}\|\Delta_{n}G\|_{L^{2}} + 2^{n(\sigma+\epsilon -\alpha)}
\|\Delta_{n}\theta\|_{L^{2}} \right) \|\widetilde{\Delta}_{n}\theta\|_{L^{\infty}} \\
\lesssim & \sum_{n\geq j-4} 2^{(j-n)(s+1-\alpha)}\left( 2^{n(s+\sigma+\epsilon -\alpha)}\|\Delta_{n}\theta\|_{L^{\infty}}\|G\|_{L^{2}}
+ 2^{n(s+\sigma +\epsilon +1 -2\alpha)}\|\Delta_{n}\theta\|_{L^{2}} \|\theta\|_{L^{\infty}} \right).
\end{align*}
Thus discrete Young's inequality yields
\begin{align*}
\|III^{1}\|_{H^{s}} \lesssim \|G\|_{L^{2}} \|\theta\|_{B_{\infty,2}^{s+\sigma+\epsilon-\alpha}} +
\|\theta\|_{L^{\infty}} \|\theta\|_{H^{s+1+\sigma+\epsilon -2\alpha}}.
\end{align*}
For $III^{2}$, similar to $III^{1}$ and using Theorem \ref{generalized riesz} we have
\begin{align*}
& 2^{js}\|\Delta_{j}III^{2}\|_{L^{2}} \\
\lesssim & 2^{js} \sum_{n\geq j-3,n\geq 0} \|\Delta_{n}u\|_{L^{2}}\|\mathcal{R}_{\alpha}\widetilde{\Delta}_{n}\theta\|_{L^{\infty}}  \\
\lesssim & 2^{js} \sum_{n\geq j-3} \left( 2^{n(\sigma + \epsilon -1)}\|\Delta_{n}G\|_{L^{2}}+2^{n(\sigma+\epsilon-\alpha)}\|\Delta_{n}\theta\|_{L^{2}} \right) 2^{n(1-\alpha)}\|\widetilde{\Delta}_{n}\theta\|_{L^{\infty}} \\
\lesssim & \sum_{n\geq j-4} 2^{(j-n)s}\left( 2^{n(s+\sigma+\epsilon-\alpha)}\|\Delta_{n}\theta\|_{L^{\infty}}\|G\|_{L^{2}}
+ 2^{n(s+\sigma+\epsilon +1 -2\alpha)}\|\Delta_{n}\theta\|_{L^{2}}\|\theta\|_{L^{\infty}} \right).
\end{align*}
Using convolution inequality, we obtain
\begin{align*}
\|III^{2}\|_{H^{s}} \lesssim \|G\|_{L^{2}}\|\theta\|_{B_{\infty,2}^{s+\sigma+\epsilon-\alpha}} + \|\theta\|_{L^{\infty}}\|\theta\|_{H^{s+1+\sigma+\epsilon-2\alpha}}.
\end{align*}
For $III^{3}$, since $\Delta_{j}III^{3} = 0$ for every $j\geq 3$, then from Bernstein's inequality we immediately have
\begin{align*}
\|III^{3}\|_{H^{s}} & \lesssim \|\mathcal{R}_{\alpha}(\Delta_{-1}u\widetilde{\Delta}_{-1}\theta)\|_{L^{2}}
+ \|\Delta_{-1}u \mathcal{R}_{\alpha}\widetilde{\Delta}_{-1}\theta\|_{L^{2}} \\
& \lesssim \|\Delta_{-1}u\widetilde{\Delta}_{-1}\theta\|_{L^{2}} + \|\Delta_{-1}u \mathcal{R}_{\alpha}\widetilde{\Delta}_{-1}\theta\|_{L^{2}} \\
& \lesssim \|G\|_{L^{2}} \|\theta\|_{\frac{2}{1-\sigma}} + \|\theta\|_{L^{2}}\|\theta\|_{L^{\frac{2}{1-\sigma}}}.
\end{align*}
Here, we used the following fact
\begin{align*}
\|\Delta_{-1}u \widetilde{\Delta}_{-1}\theta\|_{L^{2}} & \lesssim \|\Delta_{-1}u\|_{L^{q_{1}}} \|\widetilde{\Delta}_{-1}\theta\|_{L^{p_{2}}} \\
& \lesssim \|\Delta_{-1}\Lambda^{\sigma-1+\epsilon} \omega\|_{L^{q_{1}}} \|\widetilde{\Delta}_{-1}\theta\|_{L^{p_{2}}} \\
& \lesssim \|\Delta_{-1}\omega\|_{L^{p_{1}}} \|\widetilde{\Delta}_{-1}\theta\|_{L^{p_{2}}}  \\
& \lesssim \|\Delta_{-1}G\|_{L^{p_{1}}} \|\widetilde{\Delta}_{-1}\theta\|_{L^{p_{2}}} + \|\Delta_{-1}\mathcal{R}_{\alpha}\theta\|_{L^{p_{1}}}\|\widetilde{\Delta}_{-1}\theta\|_{L^{p_{2}}} \\
& \lesssim \|G\|_{L^{p_{1}}}\|\theta\|_{L^{p_{2}}} + \|\theta\|_{L^{p_{1}}}\|\theta\|_{L^{p_{2}}},
\end{align*}
where $\frac{1}{q_{1}} + \frac{1}{p_{2}} = \frac{1}{2}$ and $\frac{1}{p_{1}}+\frac{1}{p_{2}} = \frac{1}{2} + \frac{1-\sigma}{2}$. And taking $p_{1} =2$, $p_{2}=\frac{2}{1-\sigma}$ in our deduction.
From all the above statements, we can obtain our conclusion.
\end{proof}

\section{Some priori estimates}

First we need to introduce some notations. Let $G := \omega - \mathcal{R}_{\alpha}\theta$. Considering the vorticity equation
\begin{align*}
\partial_{t}\omega + u\cdot \nabla \omega + \Lambda^{\alpha} \omega = \partial_{1}\theta,
\end{align*}
and the acting of $\mathcal{R}_{\alpha}$ on the temperature equation
\begin{align*}
\partial_{t}\mathcal{R}_{\alpha}\theta + u\cdot \nabla \mathcal{R}_{\alpha}\theta + \Lambda^{\beta} \mathcal{R}_{\alpha}\theta
= -[\mathcal{R}_{\alpha},u\cdot \nabla]\theta,
\end{align*}
we directly have
\begin{align}\label{equation for G}
\partial_{t}G + u\cdot \nabla G +\Lambda^{\alpha} G =[\mathcal{R}_{\alpha},u\cdot \nabla]\theta +
\Lambda^{\beta}\mathcal{R}_{\alpha}\theta
\end{align}

\subsection{Estimation of $\|G\|_{L^{2}}$}

We present a Lemma that is proved in \cite{OTGWPOACOBNSS}, for it is useful in our proof.
\begin{lemma}\label{Lp}
Let $(\omega,\theta)$ be a smooth solution of the system (\ref{GB}). Then for every $m\in [2,\infty]$ and $t\in \mathbb{R}^{+}$
\begin{align}
\|\theta\|_{L_{t}^{m}\dot{H}^{\frac{\beta}{m}}} \lesssim \|\theta_{0}\|_{L^{2}}
\end{align}
and for $p\in [1,\infty]$
\begin{align}
\|\theta(t)\|_{L^{p}}\leq \|\theta_{0}\|_{L^{p}}.
\end{align}
\end{lemma}
The following is our estimation about $\|G\|_{L^{2}}$.
\begin{theorem}\label{GL2}
Consider (\ref{GB}) with $\sigma = 0$ and $\gamma \geq 0$. Assume that $(\omega_{0},\theta_{0})$ satisfies the
conditions in Theorem \ref{global result}. Let $(\omega,\theta)$ be the corresponding solution of (\ref{GB}), $G$ is defined as above.
Then if $\alpha$,
$\beta$ satisfies
$(\alpha,\beta) \in (\frac{3}{4},1)\times(1-\alpha,
\min{\{3\alpha-2,2-2\alpha\}})]$.

Then the following inequality holds true
\begin{align}
\|G(t)\|_{L^{2}}^{2}+\int_{0}^{t}\|G(\tau)\|_{\dot{H}^{\frac{\alpha}{2}}}^{2}\, d\tau \leq B(t)
\end{align}
where $B(t)$ is a smooth function of $t$ depending on the initial data only.
\end{theorem}
\begin{proof}
Multiplying $G$ to equation (\ref{equation for G}) and integrating with spatial variable, we obtain
\begin{align*}
\frac{1}{2}\frac{d}{dt}\|G\|_{L^{2}}^{2} + \|\Lambda^{\frac{\alpha}{2}}G\|_{L^{2}}^{2} & =
\int [\mathcal{R}_{\alpha},u\cdot \nabla]\theta \, G \, dx + \int \Lambda^{\beta-\alpha} \partial_{1}\theta \, G \, dx \\
& = I+ II
\end{align*}
For I, we have
\begin{align*}
|I| & = \left| \int \text{div} [\mathcal{R}_{\alpha},u]\theta \, G \, dx \right| \\
& \leq \|\Lambda^{\frac{\alpha}{2}}G\|_{L^{2}} \|[\mathcal{R}_{\alpha},u]\theta\|_{\dot{H}^{1-\frac{\alpha}{2}}}.
\end{align*}
Choosing $p_{3}>\frac{2}{\frac{3}{2}\alpha -1 -\epsilon}$ and using Theorem \ref{crucial commutators},
we obtain
\begin{align*}
& \|[\mathcal{R}_{\alpha},u]\theta\|_{\dot{H}^{1-\frac{\alpha}{2}}} \\
\lesssim & \|G\|_{L^{2}} \|\theta\|_{B_{\infty,2}^{1+\epsilon-\frac{3}{2}\alpha}} + \|\theta\|_{L^{\infty}}
\|\theta\|_{H^{2+\epsilon-\frac{5}{2}\alpha}}+\|G\|_{L^{2}}\|\theta\|_{L^{2}} + \|\theta\|_{L^{2}}\|\theta\|_{L^{2}}   \\
\lesssim & \|G\|_{L^{2}} \|\theta\|_{L^{p_{3}}} + \|\theta\|_{L^{\infty}}
\|\theta\|_{H^{2+\epsilon-\frac{5}{2}\alpha}}+\|G\|_{L^{2}}\|\theta\|_{L^{2}} + \|\theta\|_{L^{2}}\|\theta\|_{L^{2}},
\end{align*}
where $\epsilon>0$ is an arbitrary small number.
For II, choosing $s_{1}\in [0,\frac{\alpha}{2}]$, we have
\begin{align*}
|II| \leq \|\Lambda^{s_{1}}G\|_{L^{2}}\|\Lambda^{1+\beta -\alpha -s_{1}}\theta\|_{L^{2}}.
\end{align*}
From above statements, we can obtain
\begin{align*}
& \frac{1}{2}\frac{d}{dt}\|G\|_{L^{2}}^{2}+\|\Lambda^{\frac{\alpha}{2}}G\|_{L^{2}}^{2} \\
\lesssim & \|G\|_{L^{2}} \|\theta\|_{L^{p_{3}}} + \|\theta\|_{L^{\infty}} \|\theta\|_{H^{2+\epsilon-\frac{5}{2}\alpha}} \\
& + \|G\|_{L^{2}}\|\theta\|_{L^{2}} + \|\theta\|_{L^{2}}\|\theta\|_{L^{2}} +
\|\Lambda^{s_{1}}G\|_{L^{2}}\|\theta\|_{\dot{H}^{1+\beta-\alpha-s_{1}}}
\end{align*}
From interpolation inequality and Young's inequality, we obtain
\begin{align*}
& \|\theta\|_{\dot{H}^{1+\beta -\alpha -s_{1}}} \|G\|_{\dot{H}^{s_{1}}}   \\
\leq & C \|\theta\|_{\dot{H}^{1+\beta -\alpha -s_{1}}}\|G\|_{\dot{H}^{\frac{\alpha}{2}}}^{\frac{2s_{1}}{\alpha}}
\|G\|_{L^{2}}^{1-\frac{2s_{1}}{\alpha}} \\
\leq & C \|\theta\|_{\dot{H}^{1+\beta-\alpha-s_{1}}}^{2} + C \|G\|_{L^{2}}^{2} + \frac{1}{4} \|G\|_{\dot{H}^{\frac{\alpha}{2}}}^{2}.
\end{align*}
Using Young's inequality and the above inequality, we have
\begin{align*}
& \frac{d}{dt}\|G\|_{L^{2}}^{2}+\|\Lambda^{\frac{\alpha}{2}}G\|_{L^{2}}^{2} \\
\leq & C \|G\|_{L^{2}}^{2} + C\|\theta_{0}\|_{L^{2}\cap L^{\infty}} + C \|\theta\|_{H^{2+\epsilon -\frac{5}{2}\alpha}}^{2}
+ C\|\theta\|_{\dot{H}^{1+\beta-\alpha-s_{1}}}^{2}.
\end{align*}
Gronwall's inequality thus leads to
\begin{align*}
& \|G(t)\|_{L^{2}}^{2}+\int_{0}^{t}\|\Lambda^{\frac{\alpha}{2}}G(\tau)\|_{L^{2}}^{2}\, d\tau  \\
\leq & C e^{Ct} \left( t+\|\theta\|_{L_{t}^{2}H^{2+\epsilon-\frac{5}{2}\alpha}}^{2}+\|\theta\|^{2}_{L_{t}^{2}
\dot{H}^{1+\beta-\alpha-s_{1}}} \right).
\end{align*}
If $\frac{3}{4} < \alpha \leq \frac{4}{5}$, for $1-\alpha < \beta \leq 3\alpha -2$,
then clearly for $\epsilon>0$ small enough we have
\begin{align*}
& 0\leq 1+\beta-\frac{3}{2}\alpha \leq \frac{\beta}{2}, \\
& 0 \leq 2+\epsilon-\frac{5}{2}\alpha \leq \frac{\beta}{2},
\end{align*}
Using Lemma \ref{Lp} and interpolation inequality we easily get
\begin{align*}
\|\theta\|_{L_{t}^{2}H^{2-\frac{5}{2}\alpha+\epsilon}}^{2} +
\|\theta\|_{L_{t}^{2}\dot{H}^{1+\beta-\frac{3}{2}\alpha}}^{2}\lesssim 1+t
\end{align*}
If $\frac{4}{5} < \alpha < 1$, we choose $s_{1} = 2-2\alpha \in (0,\frac{\alpha}{2})$,
and for $1-\alpha < \beta \leq 2-2\alpha$, then
\begin{align*}
0\leq \beta-1+\alpha \leq \frac{\beta}{2}.
\end{align*}
Using Lemma \ref{Lp} and interpolation inequality we get
\begin{align*}
\|\theta\|_{L_{t}^{2}\dot{H}^{\beta-1+\alpha}}^{2} + \|\theta\|_{L_{t}^{2}H^{2 -\frac{5}{2}\alpha +\epsilon}}^{2}\lesssim 1+t.
\end{align*}
Hence, the proof is complete.
\end{proof}

\subsection{Estimation of $\|G\|_{L^{q}}$ for $q$ in suitable range}

This subsection presents the estimate of $\|G\|_{L^{q}}$ for $q$ in some suitable range. Before the main theorem, we need two Lemmas \cite{TMPATGAFTDQE,OTGWPOACOBNSS}.
\begin{lemma}\label{fractional inequality}
Suppose that $s\in [0,1]$, and $f$, $(-\Delta)^{^{s}}f \in L^{p} (\mathbb{R}^{2})$, $p\geq 2$. Then
\begin{align*}
\int_{\mathbb{R}^{2}} |f|^{p-2} f (-\Delta)^{s} f \, dx \geq \frac{2}{p} \int_{\mathbb{R}^{2}} ((-\Delta)^{\frac{s}{2}}|f|^{\frac{p}{2}})^{2} \, dx
\end{align*}
\end{lemma}
\begin{lemma}\label{iequality}
Let $\gamma \in [2,\infty)$, $s\in (0,1)$, $\alpha \in (\frac{\gamma -4}{\gamma -2},2)$. Then for every smooth function $f$
we have
\begin{align*}
\||f|^{\gamma-2}f\|_{\dot{H}^{s}} \lesssim \|f\|_{L^{\frac{2\gamma}{2-\alpha}}}^{\gamma -2}\|f\|_{\dot{H}^{s+(\frac{d}{2}-\frac{d}{\gamma})(2-\alpha)}}.
\end{align*}
\end{lemma}
The main theorem in this subsection can be stated as follows.
\begin{theorem}\label{estimateGq}
Consider (\ref{GB}) with $\sigma =0$ and $\gamma \geq 0$.
Assume that $(\omega_{0},\theta_{0})$ satisfies the conditions in Theorem $\ref{global result}$.
Let $(\omega,\theta)$ be the corresponding solution of (\ref{GB}), $G$ is defined as in (\ref{equation for G}).
If $\alpha$, $\beta$ satisfies
\begin{align*}
(\alpha,\beta) \in \left( \frac{9q-12}{8q-8},1 \right) \times
\left(1-\alpha, \min{\{ 2-2\alpha, \frac{5q-4}{3q-4}\alpha -2, \frac{1-\alpha}{\frac{4}{\alpha}(1-\frac{1}{q})-2} \}}\right)
\end{align*}
for some $q \in [2, \frac{20}{9})$.
Then for every $\widetilde{q} \in [2,q]$, we have for every $t \in \mathbb{R}^{+}$
\begin{align}
\|G(t)\|_{L^{\widetilde{q}}}^{\widetilde{q}} + \int_{0}^{t} \|G(\tau)\|_{L^{\frac{2\widetilde{q}}{2-\alpha}}}^{\widetilde{q}} \, d\tau
\leq B(t).
\end{align}
\end{theorem}
\begin{proof}
Multiplying ($\ref{equation for G}$) by $|G|^{q-2}G$ and integrating in the spatial variable we obtain for every
$s_{2},s_{3} \in (0,\frac{\alpha}{2}]$ ($s_{3}\leq s_{2}$ and both will be chosen later)
\begin{align*}
& \frac{1}{q}\frac{d}{dt}\|G(t)\|_{L^{q}}^{q} + \int_{\mathbb{R}^{2}} \Lambda^{\alpha} G |G|^{q-2}G(t) \, dx \\
\leq & \int_{\mathbb{R}^{2}} \text{div}[\mathcal{R}_{\alpha},u]\theta |G|^{q-2}G(t) \, dx
+ \int_{\mathbb{R}^{2}} \Lambda^{\beta - \alpha} \partial_{1}\theta |G|^{q-2}G(t) \, dx \\
\leq & \|[\mathcal{R}_{\alpha},u]\theta(t)\|_{\dot{H}^{1-s_{2}}} \||G|^{q-2}G(t)\|_{\dot{H}^{s_{2}}}
+ \|\theta(t)\|_{\dot{H}^{1+\beta-\alpha-s_{3}}}\||G|^{q-2}G(t)\|_{\dot{H}^{s_{3}}}.
\end{align*}
Lemma \ref{fractional inequality} and continuous embedding $\dot{H}^{\frac{\alpha}{2}}\hookrightarrow L^{\frac{4}{2-\alpha}}$
lead to
\begin{align*}
\int_{\mathbb{R}^{2}} \Lambda^{\alpha}G |G|^{q-2}G(t)\, dx \gtrsim \|G\|_{L^{\frac{2q}{2-\alpha}}}^{q}.
\end{align*}
Using Lemma \ref{iequality} we obtain
\begin{align*}
\||G|^{q-2}G\|_{\dot{H}^{s_{i}}} \lesssim \|G\|_{L^\frac{2q}{2-\alpha}}^{q-2} \|G\|_{\dot{H}^{s_{i}+(1-\frac{2}{q})(2-\alpha)}},
\quad i=2,3.
\end{align*}
From the above statements, we have
\begin{align*}
& \frac{d}{dt}\|G(t)\|_{L^{q}}^{q} + \|G\|_{L^{\frac{2q}{2-\alpha}}}^{q} \\
\lesssim & \|[\mathcal{R}_{\alpha},u]\theta(t)\|_{\dot{H}^{1-s_{2}}}\|G(t)\|_{\dot{H}^{s_{2}+(1-\frac{2}{q})(2-\alpha)}}\|G(t)\|_{L^{\frac{2q}{2-\alpha}}}^{q-2}   \\
& + \|\theta(t)\|_{\dot{H}^{1+\beta -\alpha -s_{3}}} \|G(t)\|_{\dot{H}^{s_{3}+(1-\frac{2}{q})(2-\alpha)}} \|G(t)\|_{L^{\frac{2q}{2-\alpha}}}^{q-2}.
\end{align*}
Then we choose $s_{2}$ such that $s_{2}+(1-\frac{2}{q})(2-\alpha) = \frac{\alpha}{2}$ which calls for
$s_{2} = \frac{\alpha}{2} - (1-\frac{2}{q})(2-\alpha) \in (0,\frac{\alpha}{2}]$, this is plausible if
$\alpha \in (\frac{4q-8}{3q-4},1)$ for $q \in [2,4)$. Since $s_{3} \leq s_{2}$ by interpolation we have
\begin{align*}
\|G(t)\|_{\dot{H}^{s_{3}+(1-\frac{2}{q})(2-\alpha)}} & \lesssim \|G(t)\|_{\dot{H}^{\frac{\alpha}{2}}}^{\delta}
\|G(t)\|_{L^{2}}^{1-\delta}  \\
& \lesssim B(t) \|G(t)\|_{\dot{H}^{\frac{\alpha}{2}}}^{\delta},
\end{align*}
where $\delta := \frac{2}{\alpha}\left( s_{3}+(1-\frac{2}{q})(2-\alpha) \right)$. Form the definition, we know that $\delta \leq 1$.
Also noting that if $\alpha \in (\frac{6q-8}{5q-4},1)$, we have $1-s_{2} \in (0,\alpha)$,
then form Theorem \ref{crucial commutators} and estimation of $\|G\|_{L^{2}}$, we further get
\begin{align*}
& \|[\mathcal{R}_{\alpha},u]\theta(t)\|_{H^{1-s_{2}}} \\
\lesssim & \|G(t)\|_{L^{2}}\|\theta\|_{B_{\infty,2}^{1-s_{2}-\alpha+\epsilon}} + \|\theta(t)\|_{L^{\infty}}\|\theta(t)\|_{H^{2-s_{2}-2\alpha  +\epsilon}}  \\
& + \|G(t)\|_{L^{2}}\|\theta(t)\|_{L^{2}} + \|\theta(t)\|_{L^{2}}\|\theta\|_{L^{2}}.
\end{align*}
Since $\alpha > \frac{2}{3}$ and $2\leq q < \frac{4-2\alpha}{3-\frac{5}{2}\alpha}$, we know that
$1-s_{2}-\alpha +\epsilon < 0$. Hence, we further get
\begin{align*}
\|[\mathcal{R}_{\alpha},u]\theta(t)\|_{H^{1-s_{2}}} \lesssim B(t) + \|\theta(t)\|_{H^{2-s_{2}-2\alpha +\epsilon}}.
\end{align*}
Therefore we further have
\begin{align*}
& \frac{d}{dt}\|G(t)\|_{L^{q}}^{q} + \|G(t)\|_{L^{\frac{2q}{2-\alpha}}}^{q} \\
\lesssim & \left( B(t) + \|\theta(t)\|_{H^{2-2\alpha-s_{2}+\epsilon}} \right) \|G(t)\|_{L^{\frac{2q}{2-\alpha}}}^{q-2}
\|G(t)\|_{\dot{H}^{\frac{\alpha}{2}}} \\
& + B(t) \|\theta(t)\|_{\dot{H}^{1+\beta-\alpha-s_{3}}} \|G(t)\|_{L^{\frac{2q}{2-\alpha}}}^{q-2} \|G(t)\|_{\dot{H}^{\frac{\alpha}{2}}}^{\delta}.
\end{align*}
Using Young inequality as follows
\begin{align*}
|A_{1}A_{2}A_{3}| \leq C_{1} |A_{1}|^{\frac{2q}{4-q\delta}} + C_{2} |A_{2}|^{\frac{2}{\delta}} + \frac{c}{4}|A_{3}|^{\frac{q}{q-2}},
\quad \text{for all} \, \delta \in (0,1],
\end{align*}
For $\frac{2q}{4-\delta q} \geq 2$, we have
\begin{align*}
& \frac{d}{dt}\|G(t)\|_{L^{q}}^{q} + \|G(t)\|_{L^{\frac{2q}{2-\alpha}}}^{q} \\
\lesssim & B(t) + \|\theta(t)\|_{H^{2-2\alpha -s_{2} +\epsilon}}^{\frac{2q}{4-q}} + \|G(t)\|_{\dot{H}^{\frac{\alpha}{2}}}^{2}
+ B(t) \|\theta(t)\|_{\dot{H}^{1+\beta -\alpha -s_{3}}}^{\frac{2q}{4-\delta q}},
\end{align*}
For other cases, we have
\begin{align*}
& \frac{d}{dt}\|G(t)\|_{L^{q}}^{q} + \|G(t)\|_{L^{\frac{2q}{2-\alpha}}}^{q} \\
\lesssim & B(t) + \|\theta(t)\|_{H^{2-2\alpha -s_{2} +\epsilon}}^{\frac{2q}{4-q}} + \|G(t)\|_{\dot{H}^{\frac{\alpha}{2}}}^{2}
+ \|\theta(t)\|^{2}_{\dot{H}^{1+\beta-\alpha-s_{3}}}.
\end{align*}
Integrating in time yields
\begin{align*}
& \|G(t)\|_{L^{q}}^{q} + \int_{0}^{t} \|G(\tau)\|_{L^{\frac{2q}{2-\alpha}}}^{q} \, d\tau \\
\lesssim & B(t) + \|\theta(t)\|_{L_{t}^{\frac{2q}{4-q}}H^{2-2\alpha -s_{2} +\epsilon}}^{\frac{2q}{4-q}}
+ B(t)\|\theta(t)\|_{L_{t}^{\frac{2q}{4-q\delta}}\dot{H}^{1+\beta-\alpha -s_{3}}}^{\frac{2q}{4-q\delta}},
\end{align*}
for $\frac{2q}{4-q\delta} \geq 2$, and
\begin{align*}
& \|G(t)\|_{L^{q}}^{q} + \int_{0}^{t} \|G(\tau)\|_{L^{\frac{2q}{2-\alpha}}}^{q} \, d\tau \\
\lesssim & B(t) + \|\theta(t)\|_{L_{t}^{\frac{2q}{4-q}}H^{2-2\alpha -s_{2}+\sigma +\epsilon}}^{\frac{2q}{4-q}}
+ \|\theta(t)\|^{2}_{L_{t}^{2}\dot{H}^{1+\beta -\alpha-s_{3}}},
\end{align*}
for other cases. In the above calculus we use the conclusion of Theorem \ref{GL2}, so the range of $\alpha$ and $\beta$
must satisfy the conditions in Theorem \ref{GL2}.

Let $q \in [2,\frac{20}{9})$.
If $\alpha \in (\frac{9q-12}{8q-8},\frac{8q-8}{7q-4}]$, we choose $s_{3}=s_{2}=\frac{3q-4}{2q}\alpha + \frac{4}{q} -2$,
for $\beta \in (1-\alpha,\frac{5q-4}{3q-4}\alpha - 2]$ and for small enough $\epsilon >0$, we have
\begin{align*}
& 0\leq 1+\beta -\alpha -s_{2} \leq \frac{4-q}{2q}\beta, \\
& 0\leq 2 +\epsilon -2\alpha -s_{2} \leq \frac{4-q}{2q}\beta.
\end{align*}
From Lemma \ref{Lp} and interpolation inequality we find
\begin{align*}
\|\theta\|_{L_{t}^{\frac{2q}{4-q}}\dot{H}^{1+\beta -\alpha-s_{2}}} + \|\theta\|_{L_{t}^{\frac{2q}{4-q}}H^{2+\sigma+\epsilon-2\alpha-s_{2}}} \lesssim 1+t.
\end{align*}
Let $q \in [2,\frac{20}{9})$.
If $\alpha \in (\frac{8q-8}{7q-4},1)$, we choose $s_{3}=2-2\alpha <s_{2}$, then
$\delta = \frac{2}{\alpha}\left( 2-2\alpha + \frac{q-2}{q}(2-\alpha) \right)$
and for $\beta \in \left( 1-\alpha, \min{\{ 2-2\alpha, \frac{1-\alpha}{\frac{4}{\alpha}(1-\frac{1}{q})-2} \}} \right)$
we can get
\begin{align*}
& 0\leq \beta -1 +\alpha \leq \frac{4-q\delta}{2q}\beta, \\
& 0\leq \beta -1 +\alpha \leq \frac{\beta}{2}.
\end{align*}
Hence, we have
\begin{align*}
\|\theta\|_{L_{t}^{\frac{2q}{4-q\delta}}\dot{H}^{1+\beta -\alpha -s_{3}}}
+ \|\theta\|_{L_{t}^{2}\dot{H}^{1+\beta -\alpha -s_{3}}} \lesssim 1+t.
\end{align*}
The range of $\alpha$ and $\beta$ will monotonously shrink when $q$ increase.
Hence for some $q \in [2,\frac{20}{9})$ and for every
$\widetilde{q} \in [2,q]$ we have for every $t\in \mathbb{R}^{+}$
\begin{align*}
\|G(t)\|_{L^{\widetilde{q}}}^{\widetilde{q}} + \int_{0}^{t} \|G(\tau)\|_{L^{\frac{2\widetilde{q}}{2-\alpha}}}^{\widetilde{q}} \,d\tau
\leq B(t).
\end{align*}
\end{proof}

\subsection{Estimation of $\|\omega\|_{L_{t}^{1}L^{\widetilde{q}}}$ for every $\widetilde{q}\in [2,q]$ and for some $q \in [2,\frac{20}{9})$}

In this subsection we give the estimate of $\|\omega\|_{L_{t}^{1}L^{\widetilde{q}}}$ for $\widetilde{q} \in [2,q]$
for some $q \in [2,\frac{20}{9})$.
\begin{theorem}\label{omega}
Consider (\ref{GB}) with $\sigma = 0$ and $\gamma \geq 0$. Assume that $(\omega_{0},\theta_{0})$ satisfies the conditions
in Theorem \ref{global result}. Let $(\omega,\theta)$ be the corresponding solution of (\ref{GB}). For some
$q \in [2,\frac{20}{9})$ and for all $\widetilde{q} \in [2,q]$, when $(\alpha,\beta)$ satisfies the
same conditions as in Theorem \ref{estimateGq}, we have
\begin{align*}
\|\omega\|_{L_{t}^{1}L^{\widetilde{q}}} \leq B(t).
\end{align*}
\end{theorem}
\begin{proof}
we choose $q$ as in Theorem \ref{estimateGq}. Since $\beta > 1 - \alpha$, there exists a fixed constant $\rho > 1$ such that
$\frac{\beta}{\rho} > 1-\alpha$. From the explicit formula of $G$ we have for every $\widetilde{q} \in [2,q]$
\begin{align*}
\|\omega\|_{L_{t}^{1}L^{\widetilde{q}}} & \leq \|G\|_{L_{t}^{1}L^{\widetilde{q}}} + \|\mathcal{R}_{\alpha}\theta\|_{L_{t}^{1}B_{\widetilde{q},1}^{0}} \\
& \leq B(t) + t^{1-\frac{1}{\rho}} \|\mathcal{R}_{\alpha}\theta\|_{\widetilde{L}_{t}^{\rho}B_{\widetilde{q},1}^{0}}.
\end{align*}
By a high-low frequency decomposition and a continuous embedding
$B_{\widetilde{q},\infty}^{\frac{\beta}{\rho}} \hookrightarrow B_{\widetilde{q},q}^{1-\alpha}$ we find
\begin{align*}
\|\mathcal{R}_{\alpha}\theta\|_{\widetilde{L}_{t}^{\rho}B_{\widetilde{q},1}^{0}} & \leq
\|\Delta_{-1}\mathcal{R}_{\alpha}\theta\|_{\widetilde{L}_{t}^{\rho}B_{\widetilde{q},1}^{0}} +
\|(\text{Id}-\Delta_{-1})\theta\|_{\widetilde{L}_{t}^{\rho}B_{\widetilde{q},1}^{1-\alpha}}  \\
& \lesssim \|\Delta_{-1}\theta\|_{L_{t}^{\rho}L^{\widetilde{q}}} +
\|(\text{Id}-\Delta_{-1})\theta\|_{\widetilde{L}_{t}^{\rho}B_{\widetilde{q},\infty}^{\frac{\beta}{\rho}}}  \\
& \lesssim t^{\frac{1}{\rho}} \|\theta_{0}\|_{L^{\widetilde{q}}}
+ \sup_{j \in \mathbb{N}} 2^{j\frac{\beta}{\rho}} \|\Delta_{j}\theta\|_{L_{t}^{\rho}L^{\widetilde{q}}}.
\end{align*}
Inserting the above estimate into the previous one and applying Lemma \ref{est} we obtain
\begin{align*}
\|\omega\|_{L_{t}^{1}L^{\widetilde{q}}} \leq B(t) + C t^{1-\frac{1}{\rho}}\|\omega\|_{L_{t}^{1}L^{\widetilde{q}}},
\end{align*}
where $C$ is an absolute constant depending only on $\widetilde{q},\rho$ and $\|\theta_{0}\|_{L^{2}\cap L^{\infty}}$.
If $C t^{1-\frac{1}{\rho}} = \frac{1}{2}$ equivalently, $t=(\frac{1}{2C})^{\frac{\rho}{\rho -1}}:=T_{0}$, then
for every $t\leq T_{0}$
\begin{align*}
\|\omega\|_{L_{t}^{1}L^{\widetilde{q}}} \leq B(t).
\end{align*}
Furthermore, if we evolve the system from the initial data $(u(T_{0}),\theta(T_{0}))$, then using the time
translation invariance and the fact that $\|\theta(T_{0})\|_{L^{\widetilde{q}}} \leq \|\theta_{0}\|_{L^{\widetilde{q}}}$,
we have for every $t \leq T_{0}$
\begin{align*}
\|\omega\|_{L^{1}_{[T_{0},T_{0}+t]}L^{\widetilde{q}}} \leq B(T_{0}+t).
\end{align*}
Iterating like this, we finally get for every $t \in \mathbb{R}^{+}$
\begin{align*}
\|\omega\|_{L_{t}^{1}L^{\widetilde{q}}} \leq B(t).
\end{align*}
\end{proof}

\subsection{Estimation of $\|G\|_{L_{t}^{1}B_{q,1}^{s}}$}

In this subsection, we give the estimation of $\|G\|_{L_{t}^{1}B_{q,1}^{s}}$. First we give a Lemma which is proved in \cite{ANBIAT2DQGE}
\begin{lemma}\label{Bernstein new}
Let $p \in [2,\infty)$ and $\alpha \in [0,1]$. Then there exist two positive constants $c_{p}$ and $C_{p}$ such that for
any $f \in \mathcal{S}^{'}$ and $j \in \mathbb{Z}$, we have
\begin{align*}
c_{p}2^{\frac{2\alpha j}{p}}\|\Delta_{j}f\|_{L^{p}} \leq \|\Lambda^{\alpha}(|\Delta_{j}f|^{\frac{p}{2}})\|_{L^{2}}^{\frac{2}{p}}
\leq C_{p}2^{\frac{2\alpha j}{p}} \|\Delta_{j}f\|_{L^{p}}.
\end{align*}
\end{lemma}
\begin{theorem}\label{GBq}
Consider (\ref{GB}) with $\sigma = 0$ and $\gamma \geq 0$. Assume that $(\omega_{0},\theta_{0})$ satisfies the
conditions in Theorem \ref{global result}. Let $(\omega,\theta)$ be the corresponding solution of (\ref{GB}). Let
$G$ defined as in (\ref{equation for G}). For $\alpha \in [\frac{19}{20},1)$,
$\beta \in (1-\alpha, \min\{ 2-2\alpha,\frac{8}{3}\alpha-2, $ $\frac{5 \alpha(1-\alpha)}{11-10\alpha} \})$
and $\frac{9}{10}\leq s < 2\alpha-1$. We have
\begin{align*}
\|G\|_{L_{t}^{1}B_{q,1}^{s}}\leq B(t),
\end{align*}
where $q = \frac{20}{9}-\epsilon_{1}$ for $\epsilon_{1} > 0$ is arbitrary small.
In particular,
\begin{align*}
\|G\|_{L_{t}^{1}B_{\infty,1}^{0}}\leq B(t).
\end{align*}
\end{theorem}
\begin{proof}
Applying the frequency localization operator $\Delta_{j}$ to the equation (\ref{equation for G}) we get
\begin{align*}
\partial_{t}\Delta_{j}G + u\cdot \nabla\Delta_{j}G + \Lambda^{\alpha}\Delta_{j}G
= & -[\Delta_{j},u\cdot \nabla]G - \Delta_{j}([\mathcal{R}_{\alpha},u\cdot \nabla]\theta) + \Delta_{j}\Lambda^{\beta -\alpha}\partial_{1}\theta \\
= & f_{j}(t).
\end{align*}
Multiplying the above equation by $|\Delta_{j}G|^{q-2}\Delta_{j}G$ and integrating in the spatial variable we obtain
\begin{align*}
\frac{1}{q}\frac{d}{dt}\|\Delta_{j}G(t)\|_{L^{q}}^{q} + \int_{\mathbb{R}^{2}}\Delta_{j}G|\Delta_{j}G|^{q-2}\Lambda^{\alpha}\Delta_{j}G\, dx
= \int_{\mathbb{R}^{2}}f_{j}\Delta_{j}G|\Delta_{j}G|^{q-2}.
\end{align*}
Using Lemma \ref{Bernstein new} and Lemma \ref{fractional inequality}, we can obtain
\begin{align*}
\int_{\mathbb{R}^{2}}\Lambda^{\alpha}\Delta_{j}G |\Delta_{j}G|^{q-2} \, dx \geq c2^{j\alpha} \|\Delta_{j}G\|_{L^{q}}^{q},
\end{align*}
with some positive constant $c$ independent of $j$. So we can obtain
\begin{align*}
\frac{1}{q}\frac{d}{dt}\|\Delta_{j}G(t)\|_{L^{q}}^{q} + c2^{j\alpha}\|\Delta_{j}G(t)\|_{L^{q}}^{q} \leq \|f_{j}\|_{L^{q}}
\|\Delta_{j}G\|_{L^{q}}^{q-1}.
\end{align*}
Furthermore, we have
\begin{align*}
\frac{d}{dt}\|\Delta_{j}G(t)\|_{L^{q}} + c2^{j\alpha}\|\Delta_{j}G(t)\|_{L^{q}} \leq \|f_{j}\|_{L^{q}},
\end{align*}
then
\begin{align}\label{temp inequality}
\|\Delta_{j}G\|_{L_{t}^{1}L^{q}} \lesssim & 2^{-j\alpha}\|\Delta_{j}G(0)\|_{L^{q}} + 2^{j(1+\beta-2\alpha)}\|\Delta_{j}\theta\|_{L_{t}^{1}L^{q}} \nonumber \\
& + 2^{-j\alpha}\int_{0}^{t} \|\Delta_{j}([\mathcal{R}_{\alpha},u\cdot \nabla]\theta)\|_{L^{q}} \, d\tau
+ 2^{-j\alpha} \int_{0}^{t} \|[\Delta_{j},u\cdot \nabla]G\|_{L^{q}}\, d\tau.
\end{align}
Now we deal with the second term on the right hand side of the above inequality.
For every $j \in \mathbb{N}$, by Lemma \ref{est} we have
\begin{align*}
\|\Delta_{j}\theta\|_{L_{t}^{1}L^{q}} \leq 2^{-j\beta}B(t).
\end{align*}
For the third term on the right hand side of the inequality (\ref{temp inequality}).
Using Theorem \ref{crucial commutators} we have
\begin{align*}
& 2^{-j\alpha}\int_{0}^{t} \|\Delta_{j}([\mathcal{R}_{\alpha},u\cdot \nabla]\theta)\|_{L^{q}} \, d\tau  \\
\lesssim & 2^{j(1+\epsilon -2\alpha)} \int_{0}^{t} \|[\mathcal{R}_{\alpha},u\cdot \nabla]\theta\|_{B_{q,\infty}^{\alpha-1-\epsilon}}\,d\tau   \\
\lesssim & 2^{j(1+\epsilon -2\alpha)} \int_{0}^{t} \|u\|_{\dot{B}_{q,\infty}^{1-\epsilon}}
\left( \|\theta\|_{B_{\infty,\infty}^{0}} + \|\theta\|_{L^{\infty}} \right) \, d\tau   \\
\lesssim & 2^{j(1+\epsilon -2\alpha)} \|\omega\|_{L_{t}^{1}L^{q}} \|\theta_{0}\|_{L^{\infty}} \\
\lesssim & 2^{j(1+\epsilon -2\alpha)} B(t).
\end{align*}
For the fourth term on the right hand side of the inequality (\ref{temp inequality}).
Using Lemma \ref{te} in the appendix, we have
\begin{align*}
& 2^{-j\alpha} \int_{0}^{t} \|[\Delta_{j},u\cdot \nabla]G\|_{L^{q}}\, d\tau   \\
\lesssim & 2^{j(1+\epsilon -2\alpha)}\int_{0}^{t} 2^{j(\alpha-1-\epsilon)}\|[\Delta_{j},u\cdot\nabla]G\|_{L^{q}} \, d\tau   \\
\lesssim & 2^{j(1+\epsilon-2\alpha)} \int_{0}^{t} \left( \|G\|_{B_{q,\infty}^{\alpha-1}}+\|G\|_{L^{2}}+\|\theta_{0}\|_{L^{p}\cap L^{2}} \right)\|G\|_{B_{\infty,\infty}^{0}}\, d\tau    \\
\lesssim & 2^{j(1+\epsilon-2\alpha)} \int_{0}^{t} \left( \|G\|_{L^{q}}+\|G\|_{L^{2}}+\|\theta_{0}\|_{L^{p}\cap L^{2}} \right)
\|G\|_{B_{q,1}^{s}} \, d\tau    \\
\lesssim & 2^{j(1+\epsilon-2\alpha)} \int_{0}^{t}B(\tau)\|G\|_{B_{q,1}^{s}}\, d\tau,
\end{align*}
where $s\geq \frac{2}{q}$.
Let $q \in \mathbb{N}$ be a number chosen later, then we have
\begin{align*}
& \|G\|_{L_{t}^{1}B_{q,1}^{s}}    \\
= & \sum_{m<q}2^{ms}\|\Delta_{m}G\|_{L_{t}^{1}L^{q}} + \sum_{m\geq q}2^{ms}\|\Delta_{m}G\|_{L_{t}^{1}L^{q}} \\
\lesssim & 2^{qs}B(t) + \sum_{m\geq q}2^{ms} \{ 2^{-m\alpha}\|\Delta_{m}G(0)\|_{L^{q}} + 2^{m(1-2\alpha)}B(t) \\
& + 2^{m(1+\epsilon-2\alpha)}B(t) + 2^{m(1+\epsilon-2\alpha)}B(t)\|G\|_{L_{t}^{1}B_{q,1}^{s}} \}    \\
\lesssim & 2^{qs}B(t) + \sum_{m\geq q}2^{m(s-\alpha)}\|\Delta_{m}G(0)\|_{L^{q}} + \sum_{m\geq q}2^{m(s+1-2\alpha)}B(t) \\
& + \sum_{m\geq q}2^{m(s+1+\epsilon-2\alpha)}B(t) + \sum_{m\geq q}2^{m(s+1+\epsilon-2\alpha)}B(t)\|G\|_{L_{t}^{1}B_{q,1}^{s}}.
\end{align*}
If $\alpha > \frac{2+q}{2q}$ and $\frac{2}{q}\leq s<2\alpha -1$ we can take $\epsilon > 0$ so small in the above statements
 such that $s < \alpha$ and $s+1+\epsilon-2\alpha < 0$.
From Theorem \ref{estimateGq}, we know that for $2\leq q < \frac{20}{9}$
\begin{align*}
\max{\{ \frac{9q-12}{8q-8},\frac{2+q}{2q} \}} < \alpha < 1.
\end{align*}
Through simple calculations, we easily know that the range of $\alpha$ can be the largest one when we choose $q=\frac{20}{9} - \epsilon_{1}$ for $\epsilon_{1} > 0$ is arbitrary small.
So, we have
\begin{align*}
& \alpha \in \left[\frac{19}{20},1 \right) \\
& \beta \in \left(1-\alpha, \min{\left\{ 2-2\alpha,\frac{8}{3}\,\alpha-2,\frac{5 \alpha(1-\alpha)}{11-10\alpha} \right\}}\right) \\
& \frac{9}{10}\leq s < 2\alpha-1.
\end{align*}
Hence, we finally obtain
\begin{align*}
\|G\|_{L_{t}^{1}B_{q,1}^{s}} \leq B(t)2^{qs} + 2^{-q(2\alpha -s -s-\epsilon)} B(t) \|G\|_{L_{t}^{1}B_{q,1}^{s}}.
\end{align*}
Choosing $q$ such that $2^{-q(2\alpha-s-1-\epsilon)} B(t) \approx \frac{1}{2}$.
Thus we obtain for every $t \in \mathbb{R}^{+}$
\begin{align*}
\|G\|_{L_{t}^{1}B_{q,1}^{s}}\leq B(t).
\end{align*}
By embedding this immediately leads to
\begin{align*}
\|G\|_{L_{t}^{1}B_{\infty,1}^{s-\frac{9}{10}}} \leq B(t).
\end{align*}
\end{proof}

\subsection{Estimation of $\|\omega\|_{L_{t}^{1}B_{\infty,1}^{0,\gamma}}$ and $\|\theta\|_{L_{t}^{1}B_{\infty,1}^{0,\gamma}}$}

\begin{lemma}
Consider (\ref{GB}) with $\sigma = 0$ and $\gamma \geq 0$. Assume that $(\omega_{0},\theta_{0})$ satisfies the conditions in
Theorem \ref{global result}. Let $(\omega,\theta)$ be the corresponding solution of (\ref{GB}). Take $p$ large enough such that
$\frac{2}{p}+1<\alpha+\beta$. Then we have
\begin{align*}
\|\omega\|_{L_{t}^{1}L^{p}}\leq B(t).
\end{align*}
\end{lemma}
\begin{proof}
For $\alpha +\beta >1$, we choose $\rho > 1$ such that $\frac{\beta}{\rho} > 1-\alpha$
From the definition of $G$ as in (\ref{equation for G}), we have
\begin{align*}
\|\omega\|_{L_{t}^{1}L^{p}} \leq & \|G\|_{L_{t}^{1}(B_{\infty,1}^{0}\cap L^{2})} + \|\mathcal{R}_{\alpha}\theta\|_{L_{t}^{1}B_{p,1}^{0}} \\
\leq & B(t) + t^{1-\frac{1}{\rho}} \|\mathcal{R}_{\alpha}\theta\|_{\widetilde{L}_{t}^{\rho}B_{p,1}^{0}}.
\end{align*}
Then through the same idea in the proof of Theorem \ref{omega}, we can easily get the conclusion.
\end{proof}
Now we state the main theorem in this section.
\begin{theorem}\label{88}
Consider (\ref{GB}) with $\sigma = 0$, $\gamma \geq 0$ and
$(\alpha,\beta)$ satisfies conditions as in Theorem \ref{GBq}. Assume that $(\omega_{0},\theta_{0})$ satisfies the conditions in
Theorem \ref{global result}. Let $(\omega,\theta)$ be the corresponding solution of (\ref{GB}). Then we have
\begin{align*}
& \|\omega\|_{L_{t}^{1}B_{\infty,1}^{0,\gamma}} \leq B(t) \\
& \|\theta\|_{L_{t}^{1}B_{\infty,1}^{0,\gamma}} \leq B(t).
\end{align*}
\end{theorem}
\begin{proof}
Since for $s>\frac{2}{q}$ where $q$ as in Theorem \ref{GBq}, we have
\begin{align*}
\|G\|_{B_{\infty,1}^{0,\gamma}} = & \sum_{j\geq -1}(1+|j|)^{\gamma}\|\Delta_{j}G\|_{L^{\infty}} \\
\leq & \sum_{j \geq -1}(1+|j|)^{\gamma}2^{\frac{2}{q}j}2^{-js}2^{js}\|\Delta_{j}G\|_{L^{q}} \\
\leq & C \|G\|_{B_{q,1}^{s}}.
\end{align*}
From Theorem \ref{GBq}, we obtain
\begin{align*}
\|G\|_{L_{t}^{1}B_{\infty,1}^{0,\gamma}} \leq C \|G\|_{L_{t}^{1}B_{q,1}^{s}} \leq B(t).
\end{align*}
Using the definition of $G$ as in (\ref{equation for G}), we have
\begin{align*}
\|\omega\|_{L_{t}^{1}B_{\infty,1}^{0,\gamma}} & \leq \|G\|_{L_{t}^{1}B_{\infty,1}^{0,\gamma}}+\|\mathcal{R}_{\alpha}\theta\|_{L_{t}^{1}B_{\infty,1}^{0,\gamma}}  \\
& \leq B(t) + \|\mathcal{R}_{\alpha}\theta\|_{L_{t}^{1}B_{\infty,1}^{0,\gamma}}.
\end{align*}
For the second term, we have
\begin{align*}
& \|\mathcal{R}_{\alpha}\theta\|_{L_{t}^{1}B_{\infty,1}^{0,\gamma}}   \\
= & \sum_{j \geq -1} (1+|j|)^{\gamma}\|\Delta_{j}\mathcal{R}_{\alpha}\theta\|_{L_{t}^{1}L^{\infty}} \\
\lesssim & \|\Delta_{-1}\mathcal{R}_{\alpha}\theta\|_{L_{t}^{1}L^{\infty}} + \sum_{j \geq 0}(1+|j|)^{\gamma}\|\Delta_{j}\mathcal{R}_{\alpha}\theta\|_{L_{t}^{1}L^{\infty}}    \\
\lesssim & \|\Delta_{-1}\theta\|_{L_{t}^{1}L^{\infty}} + \sum_{j=0}^{\infty}2^{j(1-\alpha)}(1+|j|)^{\gamma}\|\Delta_{j}\theta\|_{L_{t}^{1}L^{\infty}}    \\
\lesssim & t\|\theta_{0}\|_{L^{2}} + \sum_{j=0}^{\infty}2^{-j(\beta+\alpha-1-\frac{2}{p})}(1+|j|)^{\gamma}2^{j\beta}\|\Delta_{j}\theta\|_{L_{t}^{1}L^{p}}    \\
\lesssim & t\|\theta_{0}\|_{L^{2}} + \|\theta_{0}\|_{L^{p}} + \|\theta_{0}\|_{L^{\infty}}\|\omega\|_{L_{t}^{1}L^{p}}    \\
\lesssim & B(t),
\end{align*}
where $\frac{2}{p}+1<\alpha+\beta$ and we have used Lemma \ref{est} and the estimation of $\|\omega\|_{L_{t}^{1}L^{p}}$.
Hence we obtain
\begin{align*}
\|\omega\|_{L_{t}^{1}B_{\infty,1}^{0,\gamma}} \leq B(t).
\end{align*}
For $\theta$ we have
\begin{align*}
\|\theta\|_{L_{t}^{1}B_{\infty,1}^{0,\gamma}} = & \sum_{j=-1}^{\infty}(1+|j|)^{\gamma}\|\Delta_{j}\theta\|_{L_{t}^{1}L^{\infty}}    \\
\lesssim & t\|\theta_{0}\|_{L^{2}} + \sum_{j=0}^{\infty} (1+|j|)^{\gamma}2^{j\frac{2}{p}}2^{-j\beta}2^{j\beta}\|\Delta_{j}\theta\|_{L_{t}^{1}L^{p}}  \\
\lesssim & t \|\theta_{0}\|_{L^{2}} + \|\theta_{0}\|_{L^{p}} + \|\theta_{0}\|_{L^{\infty}} \|\omega\|_{L_{t}^{1}L^{p}}  \\
\lesssim & B(t),
\end{align*}
where $\frac{2}{p}+1 <\alpha+\beta$. Thus, the proof is complete.
\end{proof}

\subsection{Estimation of $\|\theta\|_{\widetilde{L}_{t}^{\infty}(H^{1-\alpha}\cap
B_{\infty,1}^{1-\alpha+\epsilon})}$, $\|\omega\|_{L_{t}^{\infty}L^{p}}$ and so on}

The following is the main result of this subsection.
\begin{theorem}
Consider (\ref{GB}) with $\sigma = 0$, $\gamma \geq 0$ and
$(\alpha,\beta)$ satisfies conditions as in Theorem \ref{GBq}. Assume that $(\omega_{0},\theta_{0})$ satisfies
the conditions in Theorem \ref{global result}. Let $(\omega,\theta)$ be the corresponding solution of (\ref{GB}).
Then for arbitrary small $\epsilon >0$ and any $p\geq 2$ we have
\begin{align*}
& \|\theta\|_{\widetilde{L}_{t}^{\infty}(H^{1-\alpha}\cap B_{\infty,1}^{1-\alpha +\epsilon})}
+ \|\theta\|_{\widetilde{L}_{t}^{1}(H^{1-\alpha +\beta}\cap B_{\infty,1}^{1-\alpha+\beta+\epsilon})} \leq B(t)  \\
& \|\omega\|_{L_{t}^{\infty}L^{p}} \leq B(t).
\end{align*}
\end{theorem}
\begin{proof}
Our proof can be divided into three steps.
Step 1: let us give the estimation of $\|u\|_{L_{t}^{1}B_{\infty,1}^{1}}$.
Since for $s>\frac{2}{q}$ where $q$ as in Theorem \ref{GBq}, take $\epsilon>0$ such that
$s-\frac{2}{q}-\epsilon >0$, we have
\begin{align*}
\|G\|_{B_{\infty,1}^{\epsilon}} & = \sum_{j\geq -1}2^{j\epsilon}\|\Delta_{j}G\|_{L^{\infty}} \\
& \leq \sum_{j\geq -1}2^{j\epsilon}2^{j\frac{2}{q}}2^{-js}2^{js}\|\Delta_{j}G\|_{L^{q}} \\
& \leq \sum_{j\geq -1}2^{-j(s-\frac{2}{q}-\epsilon)}2^{js}\|\Delta_{j}G\|_{L^{q}}   \\
& \leq C \|G\|_{B_{q,1}^{s}}.
\end{align*}
From Theorem \ref{GBq}, we obtain $\|G\|_{L_{t}^{1}B_{\infty,1}^{\epsilon}} \lesssim \|G\|_{L_{t}^{1}B_{q,1}^{s}} \lesssim B(t)$.
Using the definition of $G$ as in (\ref{equation for G}), we have
\begin{align*}
\|\omega\|_{L_{t}^{1}B_{\infty,1}^{\epsilon}} \leq \|G\|_{L_{t}^{1}B_{\infty,1}^{\epsilon}} +
\|\mathcal{R}_{\alpha} \theta\|_{L_{t}^{1}B_{\infty,1}^{\epsilon}} \leq B(t) + \|\mathcal{R}_{\alpha}\theta\|_{L_{t}^{1}B_{\infty,1}^{\epsilon}}.
\end{align*}
For the second term, we have
\begin{align*}
\|\mathcal{R}_{\alpha}\theta\|_{L_{t}^{1}B_{\infty,1}^{\epsilon}} & = \sum_{j\geq -1} 2^{j\epsilon} \|\Delta_{j}\mathcal{R}_{\alpha} \theta\|_{L_{t}^{1}L^{\infty}} \\
& \lesssim \|\Delta_{-1}\mathcal{R}_{\alpha}\theta\|_{L_{t}^{1}L^{\infty}} + \sum_{j\geq 0}2^{j\epsilon}\|\Delta_{j}\mathcal{R}_{\alpha}\theta\|_{L_{t}^{1}L^{\infty}}  \\
& \lesssim \|\Delta_{-1}\theta\|_{L_{t}^{1}L^{\infty}} + \sum_{j=0}^{\infty}2^{j(1-\alpha)}2^{j\epsilon}
\|\Delta_{j}\theta\|_{L_{t}^{1}L^{\infty}}  \\
& \lesssim t\|\theta_{0}\|_{L^{2}} + \sum_{j=0}^{\infty}2^{-j(\beta+\alpha-1-\frac{2}{p})}2^{j\epsilon}2^{j\beta}
\|\Delta_{j}\theta\|_{L_{t}^{1}L^{p}}   \\
& \lesssim t\|\theta_{0}\|_{L^{2}} + \|\theta_{0}\|_{L^{p}} + \|\theta_{0}\|_{L^{\infty}}\|\omega\|_{L_{t}^{1}L^{p}} \\
& \lesssim B(t).
\end{align*}
where $\frac{2}{p}+1+ \epsilon < \alpha +\beta$ and we have used Lemma \ref{est} and the estimation of $\|\omega\|_{L_{t}^{1}L^{p}}$.
Hence, we obtain $\|\omega\|_{L_{t}^{1}B_{\infty,1}^{\epsilon}} \leq B(t)$.
On the other hand, by Hardy-Littlewood-Sobolev inequality, we obtain
\begin{align*}
\|\Delta_{-1}u\|_{L_{t}^{1}L^{\infty}} \lesssim \|\Delta_{-1}\Lambda^{\epsilon-1}\omega\|_{L_{t}^{1}L^{\infty}}
\lesssim \|\Delta_{-1}\Lambda^{\epsilon-1}\omega\|_{L_{t}^{1}L^{\frac{2}{\epsilon}}}
\lesssim \|\omega\|_{L_{t}^{1}L^{2}} \lesssim B(t).
\end{align*}
From the above statements, we finally get
\begin{align*}
\|u\|_{L_{t}^{1}B_{\infty,1}^{1}} & \lesssim \|\Delta_{-1}u\|_{L_{t}^{1}L^{\infty}}
+ \sum_{q \in \mathbb{N}}\|\Delta_{q}\nabla u\|_{L_{t}^{1}L^{\infty}} \\
& \lesssim \|\Delta_{-1}u\|_{L_{t}^{1}L^{\infty}} + \|\omega\|_{L_{t}^{1}B_{\infty,1}^{\epsilon}} \\
& \lesssim B(t).
\end{align*}
Step 2: estimation of $\theta$.
Using Lemma \ref{expo} and the result in step 1, we can obtain
\begin{align*}
& \|\theta\|_{\widetilde{L}_{t}^{\infty}(H^{1-\alpha}\cap B_{\infty,1}^{1-\alpha+\epsilon})}
+ \|\theta\|_{\widetilde{L}_{t}^{1}(H^{1-\alpha+\beta}\cap B_{\infty,1}^{1-\alpha+\beta+\epsilon})} \\
\lesssim & \|\theta\|_{\widetilde{L}_{t}^{\infty}(H^{1-\alpha}\cap B_{\infty,1}^{1-\alpha+\epsilon})}
+ \|(\text{Id}-\Delta_{-1})\theta\|_{\widetilde{L}_{t}^{1}(H^{1-\alpha+\beta}\cap B_{\infty,1}^{1-\alpha+\beta+\epsilon})} + \|\Delta_{-1}\theta\|_{L_{t}^{1}(L^{2}\cap L^{\infty})} \\
\lesssim & e^{C \|\nabla u\|_{L_{t}^{1}L^{\infty}}}\|\theta^{0}\|_{H^{1-\alpha}\cap B_{\infty,1}^{1-\alpha+\epsilon}}
+ t\|\theta^{0}\|_{L^{2}\cap L^{\infty}} \\
\lesssim & e^{C \|u\|_{L_{t}^{1}B_{\infty,1}^{1}}} \\
\lesssim & B(t).
\end{align*}
Step 3: estimation of $\omega$.
By the equation (\ref{equation for G}) and Lemma \ref{Lpestimate}, we have
\begin{align*}
\|G(t)\|_{L^{p}} \leq \|G_{0}\|_{L^{p}} + \int_{0}^{t}\| [\mathcal{R}_{\alpha},u\cdot \nabla]\theta(\tau) \|_{L^{p}}\, d\tau
+\int_{0}^{t}\| \Lambda^{\beta}\mathcal{R}_{\alpha}\theta(\tau) \|_{L^{p}} \, d\tau.
\end{align*}
For the first integral of the RHS, using Theorem \ref{crucial commutators} with $s=0$ yields
\begin{align*}
\|[\mathcal{R}_{\alpha},u\cdot \nabla]\theta(\tau)\|_{L^{p}} & \leq
\|[\mathcal{R}_{\alpha},u\cdot \nabla]\theta(\tau)\|_{B_{p,1}^{0}}  \\
& \lesssim \|u(\tau)\|_{\dot{B}_{p,\infty}^{1-\epsilon}}
\left( \|\theta(\tau)\|_{B_{\infty,1}^{1-\alpha+\epsilon}} + \|\theta(\tau)\|_{L^{\infty}} \right) \\
& \lesssim B(\tau) \|\omega(\tau)\|_{L^{p}}.
\end{align*}
For the second integral of the RHS, we have
\begin{align*}
& \int_{0}^{1}\| \Lambda^{\beta} \mathcal{R}_{\alpha} \theta(\tau) \|_{L^{p}} \, d\tau \\
\lesssim & \|\Delta_{-1}\theta\|_{L_{t}^{1}L^{p}} + \|(\text{Id}-\Delta_{-1})\theta\|_{L_{t}^{1}B_{p,1}^{1-\alpha+\beta}} \\
\lesssim & \|\theta\|_{L_{t}^{1}L^{p}} + e^{C \|\nabla u\|_{L_{t}^{1}L^{\infty}}} \|\theta_{0}\|_{B_{p,1}^{1-\alpha}} \\
\lesssim & B(t).
\end{align*}
Hence, gathering the upper estimates we obtain
\begin{align*}
\|\omega(t)\|_{L^{p}} & \leq \|G(t)\|_{L^{p}} + \|\mathcal{R}_{\alpha}\theta(t)\|_{L^{p}} \\
& \leq B(t) + \int_{0}^{t} B(\tau)\|\omega(\tau)\|_{L^{p}} \, d\tau.
\end{align*}
Gronwall's inequality yields
\begin{align*}
\|\omega(t)\|_{L^{p}} \leq B(t).
\end{align*}
\end{proof}

At this stage, we can construct approximation system and use similar methods in \cite{GWPFABNSSWCD}
to prove the existence of the solution.


\section{Uniqueness}

In this section, we prove the uniqueness.
For convenience of the reader, we clarify some notations.
Let $(w^{1},\theta^{1})$ and $(\omega^{2},\theta^{2})$ be two solutions of system (\ref{GB}) with $\sigma =0$, $\gamma \geq 0$.
$u^{1}$ and $u^{2}$ be the corresponding velocity fields, namely
\begin{align*}
u^{j}=\nabla^{\bot}\psi^{j}, \quad \Delta \psi^{j} = (\log(\text{Id}-\Delta))^{\gamma}\omega^{j}, \quad j=1,2.
\end{align*}
Let $v^{j} = (\log(\text{Id}-\Delta))^{-\gamma}u^{j}$, $j=1,2$.
Denote
\begin{align*}
u=u^{2}-u^{1},\quad \theta=\theta^{2}-\theta^{1},\quad v=v^{2}-v^{1},\quad p=p^{2}-p^{1}.
\end{align*}

Then we give two crucial estimates
\begin{lemma}\label{finallemma1}
Assume that $\theta$ satisfies
\begin{align}\label{difference theta}
\partial_{t}\theta + u\cdot \nabla \theta^{1} + u^{2}\cdot \nabla \theta + \Lambda^{\beta}\theta = 0, \quad 0\leq \beta \leq 2.
\end{align}
Then, for any $t > 0$,
\begin{align}
\|\theta(t)\|_{B_{2,\infty}^{-1}} \leq & \|\theta(0)\|_{B_{2,\infty}^{-1}} + C \int_{0}^{t} \|v(s)\|_{L^{2}}\|\theta^{1}\|_{B_{\infty,1}^{1-\alpha,\gamma}}\, ds   \nonumber \\
& + C \int_{0}^{1} \|\omega^{2}(s)\|_{B_{\infty,1}^{0,\gamma}} \|\theta(s)\|_{B_{2,\infty}^{-\alpha}} \, ds.
\end{align}
\end{lemma}
\begin{proof}
Let $j \geq -1$. Applying $\Delta_{j}$ to (\ref{difference theta}), taking the inner product of $\Delta_{j}\theta$ with
the resulting equation and applying H\"{o}lder's inequality, we obtain
\begin{align*}
\frac{d}{dt}\|\Delta_{j}\theta\|_{L^{2}} \leq \|\Delta_{j}(u\cdot \nabla \theta^{1})\|_{L^{2}} + \|\Delta_{j}(u^{2}\cdot \nabla \theta)\|_{L^{2}}.
\end{align*}
To estimate the first term, we write
\begin{align*}
\Delta_{j}(u\cdot \nabla \theta^{1}) = J_{1} + J_{2} +J_{3},
\end{align*}
where $J_{1}$, $J_{2}$ and $J_{3}$ are given by
\begin{align*}
& J_{1} = \sum_{|j-k|\leq 2}\Delta_{j}(S_{k-1}u \nabla\Delta_{k}\theta^{1}),  \\
& J_{2} = \sum_{|j-k|\leq 2}\Delta_{j}(\Delta_{k}u \nabla S_{k-1}\theta^{1}),   \\
& J_{3} = \sum_{k \geq j-1} \Delta_{j}(\Delta_{k}u \nabla \widetilde{\Delta}_{k}\theta^{1}).
\end{align*}
$J_{1}$, $J_{2}$ and $J_{3}$ can be estimated as follows.
\begin{align*}
\|J_{1}\|_{L^{2}} \leq & C 2^{j} \|S_{j-1}u\|_{L^{2}}\|\Delta_{j}\theta^{1}\|_{L^{\infty}}  \\
\leq & C 2^{\alpha j}2^{(1-\alpha)j} \|v\|_{L^{2}} (1+|j|)^{\gamma}\|\Delta_{j}\theta^{1}\|_{L^{\infty}}    \\
\leq & C 2^{j\alpha} \|v\|_{L^{2}}\|\theta^{2}\|_{B_{\infty,1}^{1-\alpha,\gamma}}.
\end{align*}
\begin{align*}
\|J_{2}\|_{L^{2}} \leq & C \|\Delta_{j}u\|_{L^{2}} \|S_{j-1}\nabla \theta\|_{L^{\infty}}  \\
\leq & C \|\Delta_{j}v\|_{L^{2}} (1+|j|)^{\gamma} \sum_{m\leq j-2}2^{m}\|\Delta_{m}\theta^{1}\|_{L^{\infty}}    \\
\leq & C 2^{j\alpha} \|v\|_{B_{2,\infty}^{0}} \sum_{m\leq j-2}\frac{2^{m\alpha}(1+|m|)^{-\gamma}}
{2^{j\alpha}(1+|j|)^{-\gamma}}2^{(1-\alpha)m}(1+|m|)^{\gamma}\|\Delta_{m}\theta^{1}\|_{L^{\infty}}  \\
\leq & C 2^{j\alpha} \|v\|_{B_{2,\infty}^{0}} \|\theta^{1}\|_{B_{\infty,1}^{1-\alpha,\gamma}}.
\end{align*}
\begin{align*}
\|J_{3}\|_{L^{2}} \leq &  C2^{j}\sum_{k \geq j-1} (1+|k|)^{\gamma} \|\Delta_{k}\theta^{1}\|_{L^{\infty}}
\|\Delta_{k}v\|_{L^{2}}    \\
\leq & C2^{j\alpha} 2^{j(1-\alpha)} \sum_{k\geq j-1} 2^{-k(1-\alpha)} (1+|k|)^{\gamma}\|\Delta_{k}\theta^{1}\|_{L^{\infty}}
2^{k(1-\alpha)}\|\Delta_{k}v\|_{L^{2}}  \\
\leq & C 2^{j\alpha}\sum_{k \geq j-1}2^{(j-k)(1-\alpha)}(1+|k|)^{\gamma}\|\Delta_{k}\theta^{1}\|_{L^{\infty}}
2^{k(1-\alpha)}\|\Delta_{k}v\|_{L^{2}}  \\
\leq & C 2^{j\alpha} \|v\|_{B_{2,\infty}^{0}} \|\theta^{1}\|_{B_{\infty,1}^{1-\alpha,\gamma}}.
\end{align*}
To estimate the second term, we write
\begin{align}\label{5.5}
\Delta_{j}(u^{2}\cdot \nabla \theta) = K_{1} + K_{2} + K_{3} + K_{4} + K_{5},
\end{align}
where
\begin{align*}
& K_{1} = \sum_{|j-k|\leq 2}[\Delta_{j},S_{k-1}u^{2}\cdot \nabla]\Delta_{k}\theta,  \\
& K_{2} = \sum_{|j-k|\leq 2}(S_{k-1}u^{2}-S_{j}u^{2})\cdot \nabla\Delta_{j}\Delta_{k}\theta,    \\
& K_{3} = S_{j}u^{2}\cdot \nabla\Delta_{j}\theta,   \\
& K_{4} = \sum_{|j-k|\leq 2}\Delta_{j}(\Delta_{k}u^{2}\cdot \nabla S_{k-1}\theta),  \\
& K_{5} = \sum_{k \geq j-1} \Delta_{j}(\Delta_{k}u^{2}\cdot \nabla \widetilde{\Delta}_{k}\theta).
\end{align*}
Since $\nabla \cdot u^{2} = 0$, we know that
\begin{align*}
\int \Delta_{j}\theta K_{3}\, dx =0
\end{align*}
By a standard commutator estimate, we obtain
\begin{align*}
\|K_{1}\|_{L^{2}} \leq & C \|x \Phi_{j}(x)\|_{L^{1}}\|\nabla S_{j-1} u^{2}\|_{L^{\infty}} \|\nabla \Delta_{j}\theta\|_{L^{2}}   \\
\leq & C \|x\Phi_{j}(x)\|_{L^{1}}\|\omega^{2}\|_{B_{\infty,1}^{0,\gamma}} \|\Delta_{j}\theta\|_{L^{2}}.
\end{align*}
where $\Phi_{j}(x)$ is the kernel of the operator $\Delta_{j}$.
For $j \geq j_{0}$ with $j_{0} = 2$, we apply Bernstein's inequality to obtain
\begin{align*}
\|K_{2}\|_{L^{2}} \leq & C \|\Delta_{j}u^{2}\|_{L^{\infty}} \|\nabla \Delta_{j}\theta\|_{L^{2}} \\
\leq & C \|\Delta_{j}\nabla u^{2}\|_{L^{\infty}} \|\Delta_{j}\theta\|_{L^{2}}   \\
\leq & C \|\omega^{2}\|_{B_{\infty,1}^{0,\gamma}} \|\Delta_{j}\theta\|_{L^{2}}.
\end{align*}
Again, for $j \geq j_{0}$ with $j_{0} = 2$, we have
\begin{align*}
\|K_{4}\|_{L^{2}} \leq & C \|\Delta_{j}u^{2}\|_{L^{\infty}} \|S_{j-1}\nabla \theta\|_{L^{2}}    \\
\leq & C 2^{j\alpha} \|\Delta_{j}\nabla u^{2}\|_{L^{\infty}} \sum_{m\leq j-2} 2^{(1+\alpha)(m-j)}2^{-m\alpha} \|\Delta_{m}\theta\|_{L^{2}}    \\
\leq & C 2^{j\alpha} \|\omega^{2}\|_{B_{\infty,1}^{0,\gamma}} \|\theta\|_{B_{2,\infty}^{-\alpha}}.
\end{align*}
\begin{align*}
\|K_{5}\|_{L^{2}} \leq & C 2^{j} \sum_{k\geq j-1} \|\Delta_{k}u^{2}\|_{L^{\infty}} \|\Delta_{k}\theta\|_{L^{2}} \\
\leq & C 2^{j\alpha} \sum_{k \geq j-1} 2^{-k} 2^{k(1-\alpha)} \|\Delta_{k}\nabla u^{2}\|_{L^{\infty}} \|\Delta_{k}\theta\|_{L^{2}}  \\
\leq & C 2^{j\alpha} \sum_{k \geq j-1} \|\Delta_{k}\nabla u^{2}\|_{L^{\infty}} 2^{-k\alpha} \|\Delta_{k}\theta\|_{L^{2}}    \\
\leq & 2^{j\alpha} \|\omega^{2}\|_{B_{\infty,1}^{0,\gamma}} \|\theta\|_{B_{2,\infty}^{-\alpha}}.
\end{align*}
From all the above estimates, we obtain
\begin{align*}
\frac{d}{dt}\|\Delta_{j}\theta\|_{L^{2}} \leq C 2^{j\alpha} \|v\|_{L^{2}}\|\theta^{1}\|_{B_{\infty,1}^{1-\alpha,\gamma}}
+ C \|\omega^{2}\|_{B_{\infty,1}^{0,\gamma}} \|\Delta_{j}\theta\|_{L^{2}} + C 2^{j\alpha}\|\omega^{2}\|_{B_{\infty,1}^{0,\gamma}}
\|\theta\|_{B_{2,\infty}^{-\alpha}}.
\end{align*}
Integrating in time leads to
\begin{align*}
\|\Delta_{j}\theta(t)\|_{L^{2}} \leq & \|\Delta_{j}\theta(0)\|_{L^{2}} + C2^{j\alpha} \int_{0}^{t}\|v(s)\|_{L^{2}}
\|\theta^{1}(s)\|_{B_{\infty,1}^{1-\alpha,\gamma}}\, ds \\
& + C 2^{j\alpha}\int_{0}^{t}\|\omega^{2}(s)\|_{B_{\infty,1}^{0,\gamma}}\|\theta(s)\|_{B_{2,\infty}^{-\alpha}}\, ds
\end{align*}
Hence, we finally get
\begin{align*}
\|\theta(t)\|_{B_{2,\infty}^{-\alpha}} \leq & \|\theta(0)\|_{B_{2,\infty}^{-\alpha}} + C\int_{0}^{t}\|v(s)\|_{L^{2}}
\|\theta^{1}(s)\|_{B_{\infty,1}^{1-\alpha,\gamma}}\, ds \\
& + C \int_{0}^{t}\|\omega^{2}(s)\|_{B_{\infty,1}^{0,\gamma}}\|\theta(s)\|_{B_{2,\infty}^{-\alpha}} \, ds.
\end{align*}
\end{proof}

\begin{lemma}
Assume that $v$ satisfies
\begin{align}\label{vvv}
\partial_{t} v + u^{2}\cdot \nabla v + u\cdot \nabla v^{1} - \sum_{j=1}^{2} \left( u_{j}^{2}\nabla v_{j}
+u_{j}\nabla v_{j}^{1} \right) + \Lambda^{\alpha}v = -\nabla p + \theta e_{2},
\end{align}
for $\alpha \in (0,1]$. Then
\begin{align*}
\|v(t)\|_{B_{2,\infty}^{0}} \leq & \|v(0)\|_{B_{2,\infty}^{0}} + \sup_{0\leq s\leq t} \|\theta(s)\|_{B_{2,\infty}^{-\alpha}} \nonumber   \\
& + C\int_{0}^{t} \|v(s)\|_{L^{2}} \left( \|\omega^{1}(s)\|_{B_{\infty,1}^{0,\gamma}} + \|\omega^{2}(s)\|_{B_{\infty,1}^{0,\gamma}} \right)\, ds.
\end{align*}
\end{lemma}
\begin{proof}
Let $j \geq -1$. After applying $\Delta_{j}$ to equation (\ref{vvv}), taking the inner product with $\Delta_{j}v$
and integrating by parts, we find
\begin{align*}
\frac{1}{2}\frac{d}{dt}\|\Delta_{j}v\|_{L^{2}}^{2} + c2^{j\alpha}\|\Delta_{j}v\|_{L^{2}}^{2}
=L_{1}+L_{2}+L_{3}+L_{4}+L_{5},
\end{align*}
where
\begin{align*}
& L_{1}=-\int \Delta_{j}v\cdot \Delta_{j}(u^{2}\cdot \nabla v)\, dx   \\
& L_{2}=-\int \Delta_{j}v\cdot \Delta_{j}(u\cdot \nabla v^{1})\, dx   \\
& L_{3}=-\sum_{n=1}^{2} \int \Delta_{j}v\cdot \Delta_{j}(u_{n}^{2}\nabla v_{n})\, dx \\
& L_{4}=-\sum_{n=1}^{2} \int \Delta_{j}v\cdot \Delta_{j}(u_{n}\nabla v_{n}^{1})\, dx \\
& L_{5}=-\int \Delta_{j}v_{2}\cdot \Delta_{j}\theta.
\end{align*}
To estimate $L_{1}$, we decompose $\Delta_{j}(u^{2}\cdot \nabla v)$ as in (\ref{5.5}) and bound the
components in a similar fashion as the above Lemma. We obtain after applying H\"{o}lder's inequality
\begin{align*}
|L_{1}|\leq C \|\Delta_{j}v\|_{L^{2}}\|v\|_{L^{2}}\|\omega^{2}\|_{B_{\infty,1}^{0,\gamma}}.
\end{align*}
To deal with $L_{2}$, similar to the proof in the above Lemma we obtain
\begin{align*}
|L_{2}|\leq C\|\Delta_{j}v\|_{L^{2}}\|v\|_{L^{2}}\|\omega^{1}\|_{B_{\infty,1}^{0,\gamma}}.
\end{align*}
To handle $L_{3}$, we integrate by part and use the divergence-free condition to obtain
\begin{align*}
L_{3}=\sum_{n=1}^{2}\int \Delta_{j}v \cdot \Delta_{j} (v_{n}\nabla u_{n}^{2}).
\end{align*}
Then using the same idea as in Lemma \ref{finallemma1} we have
\begin{align*}
|L_{3}|\leq C \|\Delta_{j}v\|_{L^{2}}\|v\|_{L^{2}}\|\omega^{2}\|_{B_{\infty,1}^{0,\gamma}}.
\end{align*}
We can easily notice that $L_{4}$ admits the same bound as $L_{2}$.
$L_{5}$ can be bounded by applying H\"{o}lder's inequality
\begin{align*}
|L_{5}|\leq \|\Delta_{j}v\|_{L^{2}}\|\Delta_{j}\theta\|_{L^{2}} \leq 2^{j\alpha}\|\Delta_{j}v\|_{L^{2}}
\|\theta\|_{B_{2,\infty}^{-\alpha}}.
\end{align*}
From all the above statements, we find
\begin{align*}
& \frac{d}{dt}\|\Delta_{j}v\|_{L^{2}} + 2^{j\alpha}\|\Delta_{j}v\|_{L^{2}}    \\
\leq & C \|v\|_{L^{2}}\left( \|\omega^{1}\|_{B_{\infty,1}^{0,\gamma}}+\|\omega^{2}\|_{B_{\infty,1}^{0,\gamma}} \right)+2^{j\alpha}\|\theta\|_{B_{2,\infty}^{-\alpha}}.
\end{align*}
Integrating in time yields
\begin{align*}
\|\Delta_{j}v(t)\|_{L^{2}} \leq & e^{-2^{j\alpha}t} \|\Delta_{j}v(0)\|_{L^{2}}
+ \int_{0}^{t} e^{-2^{j\alpha}(t-s)} 2^{j\alpha} \|\theta(s)\|_{B_{2,\infty}^{-\alpha}} \, ds   \\
& + C \int_{0}^{t} e^{-2^{j\alpha}(t-s)} \|v(s)\|_{L^{2}}
\left( \|\omega^{1}(s)\|_{B_{\infty,1}^{0,\gamma}} + \|\omega^{2}(s)\|_{B_{\infty,1}^{0,\gamma}} \right) \, ds.
\end{align*}
Therefore,
\begin{align*}
\|v(t)\|_{B_{2,\infty}^{0}} \leq & \|v(0)\|_{B_{2,\infty}^{0}} + \sup_{0\leq s\leq t}\|\theta(s)\|_{B_{2,\infty}^{-\alpha}}    \\
& + C \int_{0}^{t} \|v(s)\|_{L^{2}} \left( \|\omega^{1}(s)\|_{B_{\infty,1}^{0,\gamma}}
+\|\omega^{2}(s)\|_{B_{\infty,1}^{0,\gamma}} \right) \, ds.
\end{align*}
This completes the proof.
\end{proof}

At the end, we give the main theorem of this section.
\begin{theorem}
Assume that $(\omega_{0},\theta_{0})$ satisfies the conditions stated in Theorem \ref{global result}.
Let $\sigma = 0$, $\gamma \geq 0$ and $q > 2$. Let $(\omega^{1},\theta^{1})$ and $(\omega^{2}, \theta^{2})$ be
two solutions of (\ref{GB}) satisfying for any $t > 0$,
\begin{align*}
\omega^{1},\omega^{2} \in L_{t}^{1}L^{2}\cap L_{t}^{1}B_{\infty,1}^{0,\gamma},\quad \theta^{1},\theta^{2}\in L_{t}^{1}L^{2}\cap L_{t}^{1}B_{\infty,1}^{0,\gamma}.
\end{align*}
Then they must coincide.
\end{theorem}
\begin{proof}
Using the notations stated in the beginning of this section, we know that
$v$, $\theta$, $u$ and $p$ satisfy (\ref{difference theta}) and (\ref{vvv}).
In our deduction, we will use the following two inequalities
\begin{align*}
\|v\|_{L^{2}} \leq C \|v\|_{B_{2,\infty}^{0}} \log \left( 1 + \frac{\|v\|_{H^{1}}}{\|v\|_{B_{2,\infty}^{0}}} \right),
\end{align*}
and
\begin{align*}
\|v\|_{H^{1}} \leq \|\omega^{1}\|_{L^{2}} + \|\omega^{2}\|_{L^{2}}.
\end{align*}
Combining the inequalities above and setting
\begin{align*}
Y(t) = \|\theta(t)\|_{B_{2,\infty}^{-\alpha}} + \|v(t)\|_{B_{2,\infty}^{0}},
\end{align*}
we obtain
\begin{align*}
Y(t) \leq 2Y(0) + C \int_{0}^{t} D_{1}(s)\left[ Y(s)\log \left( 1+\frac{D_{2}(s)}{Y(s)} \right) + Y(s) \right] \, ds,
\end{align*}
where
\begin{align*}
& D_{1}(s) = \|\theta^{1}(s)\|_{B_{\infty,1}^{1-\alpha,\gamma}}+\|\omega^{1}(s)\|_{B_{\infty,1}^{0,\gamma}}+\|\omega^{2}(s)\|_{B_{\infty,1}^{0,\gamma}}  \\
& D_{2}(s) = \|\omega^{2}(s)\|_{L^{2}}+\|\omega^{2}(s)\|_{L^{2}}.
\end{align*}
Using the same idea in the proof of the integrable of $\|\theta\|_{B_{\infty,1}^{0,\gamma}}$, we can prove the
that $\|\theta^{1}(s)\|_{B_{\infty,1}^{1-\alpha,\gamma}}$ is integrable.
Hence, we know that $D_{1}$ and $D_{2}$ are integrable.
By Osgood's inequality we get $Y(t)=0$.
This completes the proof.
\end{proof}


\section{Appendix: Technical Lemmas}

Here we give some useful estimates in Besov framework.
\begin{lemma}\label{te}
Let $u$ be a smooth divergence-free vector field of $\mathbb{R}^{d}$ and $f$ be a smooth scalar function. Then

(1) for every $\alpha \in (\sigma+\epsilon,1)$ and $p \in [2,\infty]$
\begin{align*}
& \sup_{q\geq -1} 2^{q(\alpha -1 -\epsilon -\sigma)}\|[\Delta_{q},u\cdot \nabla]f\|_{L^{p}} \\
\lesssim & \left( \|\Lambda^{1-\sigma-\epsilon}u\|_{B_{p,\infty}^{\alpha-1}} + \|u\|_{L^{2}} \right) \|f\|_{B_{\infty,\infty}^{0}}
\end{align*}

(2) for a special $u = \nabla^{\bot}\Delta^{-1}\Lambda^{\sigma}(\log(\text{Id}-\Delta))^{\gamma}\omega$
\begin{align*}
& \sup_{q\geq -1}2^{q(\alpha -1-\epsilon -\sigma)}\|[\Delta_{q},u\cdot \nabla]f\|_{L^{p}} \\
\lesssim & \left( \|G\|_{B_{p,\infty}^{\alpha -1}} + \|G\|_{L^{2}} + \|\theta_{0}\|_{L^{p}\cap L^{2}} \right) \|f\|_{B_{\infty,\infty}^{0}}.
\end{align*}
\end{lemma}
\begin{proof}
(1)
From Bony's decomposition we have
\begin{align*}
[\Delta_{q},u\cdot \nabla]f = & \sum_{|j-q|\leq 4} [\Delta_{q},S_{q-1}u\cdot\nabla]\Delta_{j}f
+ \sum_{|j-q|\leq 4} [\Delta_{q},\Delta_{j}u\cdot \nabla]S_{j-1}f \\
& + \sum_{j \geq q-3, 1\leq i \leq n} [\Delta_{q}\partial_{i},\Delta_{j}u^{i}]\widetilde{\Delta}_{j}f \\
= & I_{q} + II_{q} + III_{q}.
\end{align*}
Estimation of $I_{q}$.
Since $\Delta_{q}:=h_{q}(\cdot)* = 2^{qd}h(2^{q}\cdot)*$ with $h \in \mathcal{S}(\mathbb{R}^{d})$, then from Lemma \ref{commutator}
we get for every $\alpha < 1$
\begin{align*}
\|I_{q}\|_{L^{p}} \lesssim & \sum_{|j-q|\leq 4}\||x|^{1-\sigma-\epsilon}\phi_{q}\|_{L^{1}}\|\Lambda^{1-\sigma-\epsilon}S_{j-1}u\|_{L^{p}}2^{j}\|\Delta_{j}f\|_{L^{\infty}} \\
\lesssim & \|f\|_{B_{\infty,\infty}^{0}}\||x|^{1-\sigma-\epsilon}\phi\|_{L^{1}}
\sum_{|j-q|\leq 4}2^{(j-q)(1-\sigma-\epsilon)}2^{j(1+\epsilon+\sigma-\alpha)}  \\
& \times \sum_{k\leq j-2}2^{(j-k)(\alpha -1)}2^{k(\alpha-1)}
\|\Delta_{k}\Lambda^{1-\sigma-\epsilon}u\|_{L^{p}} \\
\lesssim & 2^{q(1+\epsilon+\sigma-\alpha)}\|\Lambda^{1-\sigma-\epsilon}u\|_{B_{p,\infty}}^{\alpha-1} \|f\|_{B_{\infty,\infty}^{0}},
\end{align*}
thus we have
\begin{align*}
\sup_{q\geq -1} 2^{q(\alpha -1 -\epsilon-\sigma)}\|I_{q}\|_{L^{p}}\lesssim \|\Lambda^{1-\sigma-\epsilon}u\|_{B_{p,\infty}^{\alpha -1}}
\|f\|_{B_{\infty,\infty}^{0}}.
\end{align*}
Estimation of $II_{q}$. Similar to the estimation of $I_{q}$, we can obtain
\begin{align*}
\|II_{q}\|_{L^{p}} \lesssim & \sum_{|j-q|\leq 4}\||x|^{1-\sigma-\epsilon}\phi_{q}\|_{L^{1}}\|\Lambda^{1-\sigma-\epsilon}\Delta_{j}u\|_{L^{p}}\|S_{j-1}\nabla f\|_{L^{\infty}} \\
\lesssim & \|\Lambda^{1-\sigma -\epsilon} u\|_{B_{p,\infty}^{\alpha-1}} \sum_{|j-q|\leq 4}\||x|^{1-\sigma-\epsilon}\phi_{q}\|_{L^{1}}
2^{j(1-\alpha)} \|S_{j-1}\nabla f\|_{L^{\infty}} \\
\lesssim & \|\Lambda^{1-\sigma-\epsilon}u\|_{B_{p,\infty}^{\alpha -1}}
\sum_{|j-q|\leq 4}2^{-q(1-\sigma -\epsilon)}2^{j(1-\alpha)}\sum_{k\leq j-2}2^{k}\|\Delta_{k}f\|_{L^{\infty}} \\
\lesssim & 2^{q(1-\alpha+\sigma+\epsilon)}\|\Lambda^{1-\sigma-\epsilon}u\|_{B_{p,\infty}^{\alpha-1}}\|f\|_{B_{\infty,\infty}^{0}},
\end{align*}
thus
\begin{align*}
\sup_{q \geq -1}2^{q(\alpha -1 -\sigma-\epsilon)} \|II_{q}\|_{L^{p}} \lesssim \|\Lambda^{1-\sigma-\epsilon}u\|_{B_{p,\infty}^{\alpha-1}} \|f\|_{B_{\infty,\infty}^{0}}.
\end{align*}
Estimation of $III_{q}$. We further write
\begin{align*}
III_{q} & = \sum_{j\geq q-3,j\in \mathbb{N},1\leq i\leq d} [\Delta_{q}\partial_{i},\Delta_{j}u^{i}]\widetilde{\Delta}_{j}f
+ \sum_{1\leq i\leq d}[\Delta_{q}\partial_{i},\Delta_{-1}u^{i}]\widetilde{\Delta}_{-1}f \\
& = III_{q}^{1} + III_{q}^{2}.
\end{align*}
For the first term, we get for every $\alpha > 0$
\begin{align*}
\|III_{q}^{1}\|_{L^{p}} \leq & \sum_{j\geq q-3,j\in \mathbb{N},1\leq i\leq d}
\|\partial_{i}\Delta_{q}(\Delta_{j}u^{i})\widetilde{\Delta}_{j}f\|_{L^{p}}
+ \sum_{j\geq q-3,j\in \mathbb{N},1\leq i\leq d}\|\Delta_{j}u^{i}\partial_{i}\Delta_{q}\widetilde{\Delta}_{j}f\|_{L^{p}} \\
\lesssim & 2^{q(1+\epsilon +\sigma -\alpha)}\sum_{j\geq q-3,j\in \mathbb{N}}2^{(q-j)(\alpha -\epsilon -\sigma)}
2^{j(\alpha-1)}\|\Delta_{j}\Lambda^{1-\sigma-\epsilon}u\|_{L^{p}}\|\widetilde{\Delta}_{j}f\|_{L^{\infty}} \\
\lesssim & 2^{q(1-\alpha+\sigma+\epsilon)} \|\Lambda^{1-\sigma-\epsilon}u\|_{B_{p,\infty}^{\alpha-1}}\|f\|_{B_{\infty,\infty}^{0}},
\end{align*}
thus
\begin{align*}
\sup_{q\geq -1}2^{q(\alpha -1-\epsilon-\sigma)}\|III_{q}^{1}\|_{L^{p}}\lesssim \|\Lambda^{1-\sigma-\epsilon}u\|_{B_{p,\infty}^{\alpha-1}}\|f\|_{B_{\infty,\infty}^{0}}.
\end{align*}
For the second term, due to $III_{q}^{2} = 0$ for every $q\geq 3$, we get for $p\geq 2$
\begin{align*}
& \sup_{q \geq -1}2^{q(\alpha-1-\epsilon-\sigma)}\|III_{q}^{2}\|_{L^{p}} \\
= & \sup_{q\geq -1}2^{q(\alpha-1-\epsilon-\sigma)}\|[\Delta_{q}\partial,\Delta_{-1}u]\widetilde{\Delta}_{-1}f\|_{L^{p}} \\
\lesssim & \|u\|_{L^{2}}\|f\|_{B_{\infty,\infty}^{0}}.
\end{align*}
Hence, the proof is complete.

(2)
Using Bony's decomposition, we obtain same formula as in the proof of (1). Estimation of $I_{q}$.
Similar to the proof of (1), we have
\begin{align*}
\|I_{q}\|_{L^{p}} \lesssim & \|f\|_{B_{\infty,\infty}^{0}} \sum_{|j-q|\leq 4}2^{(j-q)(1-\sigma-\epsilon)}2^{j(1+\epsilon+\sigma-\alpha)}   \\
& \times \sum_{k\leq j-2}2^{(j-k)(\alpha-1)}2^{k(\alpha-1)}
\left( \|\Delta_{k}G\|_{L^{p}} + \|\Delta_{k}\mathcal{R}_{\alpha}\theta\|_{L^{p}} \right) \\
\lesssim & 2^{q(1+\epsilon+\sigma-\alpha)}\|G\|_{B_{p,\infty}^{\alpha-1}}\|f\|_{B_{\infty,\infty}^{0}} \\
& + \|f\|_{B_{\infty,\infty}^{0}}\sum_{|j-q|\leq 4} 2^{(j-q)(1-\sigma -\epsilon)}2^{j(1+\epsilon+\sigma-\alpha)}
\sum_{k\leq j-2}2^{(j-k)(\alpha-1)}\|\theta\|_{L^{p}} \\
\lesssim & 2^{q(1+\epsilon+\sigma-\alpha)} \|f\|_{B_{\infty,\infty}^{0}} \left( \|G\|_{B_{p,\infty}^{\alpha-1}} + \|\theta\|_{L^{p}} \right),
\end{align*}
thus
\begin{align*}
\sup_{q\geq -1} 2^{q(\alpha-1-\epsilon-\sigma)} \|I_{q}\|_{L^{p}} \lesssim \left( \|G\|_{B_{p,\infty}^{\alpha-1}} + \|\theta\|_{L^{p}} \right) \|f\|_{B_{\infty,\infty}^{0}}.
\end{align*}
For $II_{q}$, we have
\begin{align*}
\|II_{q}\|_{L^{p}} \lesssim & \sum_{|j-q|\leq 4}\||x|^{1-\sigma-\epsilon}\phi_{q}\|_{L^{1}}
\|\Delta_{j}\Lambda^{1-\sigma-\epsilon}u\|_{L^{p}}\|S_{j-1}\nabla f\|_{L^{\infty}} \\
\lesssim & \sum_{|j-q|\leq 4}2^{-q(1-\sigma-\epsilon)}\|\Delta_{j}\Lambda^{1-\sigma-\epsilon}u\|_{L^{p}}\|S_{j-1}\nabla f\|_{L^{\infty}} \\
\lesssim & \sum_{|j-q|\leq 4}2^{-q(1-\sigma-\epsilon)}\left(\|\Delta_{j}G\|_{L^{p}}+\|\Delta_{j}\mathcal{R}_{\alpha}\theta\|_{L^{p}}\right)\sum_{k\leq j-2}2^{k}\|\Delta_{k}f\|_{L^{\infty}} \\
\lesssim & \sum_{|j-q|\leq 4}2^{-q(1-\sigma-\epsilon)}2^{j(1-\alpha)}\sum_{k\leq j-2}2^{k}\|\Delta_{j}f\|_{L^{\infty}}
\left( \|G\|_{B_{p,\infty}^{\alpha-1}} + \|\theta\|_{L^{p}} \right) \\
\lesssim & 2^{q(1-\alpha+\epsilon+\sigma)}\left( \|G\|_{B_{p,\infty}^{\alpha-1}} + \|\theta\|_{L^{p}} \right)\|f\|_{B_{\infty,\infty}^{0}},
\end{align*}
thus
\begin{align*}
\sup_{q\geq -1} 2^{q(\alpha-1-\epsilon-\sigma)} \|I_{q}\|_{L^{p}} \lesssim \left( \|G\|_{B_{p,\infty}^{\alpha-1}} + \|\theta\|_{L^{p}} \right) \|f\|_{B_{\infty,\infty}^{0}}.
\end{align*}
For the term $III_{q}^{1}$. We can calculus as follows.
\begin{align*}
\|III_{q}^{1}\|_{L^{p}} \lesssim & 2^{q(1-\alpha+\sigma+\epsilon)} \sum_{j\geq q-3, j\in \mathbb{N}} 2^{(q-j)(\alpha-\sigma-\epsilon)}
2^{j(\alpha-1)}\|\Delta_{j}\Lambda^{1-\sigma-\epsilon}u\|_{L^{p}}\|\widetilde{\Delta}_{j}f\|_{L^{\infty}} \\
\lesssim & 2^{q(1-\alpha+\sigma+\epsilon)}\sum_{j\geq q-3,j\in \mathbb{N}}2^{(q-j)(\alpha-\epsilon-\sigma)}2^{j(\alpha-1)}
\|\Delta_{j}G\|_{L^{p}}\|\widetilde{\Delta}_{j}f\|_{L^{\infty}} \\
& + 2^{q(1-\alpha+\sigma+\epsilon)}\sum_{j\geq q-3,j\in \mathbb{N}} 2^{(q-j)(\alpha-\sigma-\epsilon)}2^{j(\alpha-1)} \|\Delta_{j}\mathcal{R}_{\alpha}\theta\|_{L^{p}}\|\widetilde{\Delta}_{j}f\|_{L^{\infty}}    \\
\lesssim & 2^{q(1-\alpha+\sigma+\epsilon)}\left( \|G\|_{B_{p,\infty}^{\alpha-1}} + \|\theta\|_{L^{p}} \right) \|f\|_{B_{\infty,\infty}^{0}}.
\end{align*}
For the term $III_{q}^{2}$. For every $q\geq 3$ we know that $III_{q}^{2} = 0$. So for $p \geq 2$, we have
\begin{align*}
& \sup_{q \geq -1}2^{q(\alpha-1-\epsilon-\sigma)}\|III_{q}^{2}\|_{L^{p}}  \\
= & \sup_{q \geq -1} 2^{q(\alpha-1-\sigma-\epsilon)}\|[\Delta_{q}\partial,\Delta_{-1}u]\widetilde{\Delta}_{-1}f\|_{L^{p}}   \\
\lesssim & \|\Delta_{-1}u\|_{L^{p}}\|f\|_{B_{\infty,\infty}^{0}}    \\
\lesssim & \|\Delta_{-1}\Lambda^{\sigma+\epsilon-1}u\|_{L^{p}}\|f\|_{B_{\infty,\infty}^{0}} \\
\lesssim & \left( \|\Delta_{-1}G\|_{L^{p}} + \|\Delta_{-1}\mathcal{R}_{\alpha}\theta\|_{L^{p}} \right) \|f\|_{B_{\infty,\infty}^{0}}    \\
\lesssim & \left( \|G\|_{L^{2}} + \|\theta_{0}\|_{L^{2}} \right) \|f\|_{B_{\infty,\infty}^{0}}.
\end{align*}
From all the above statements, we can obtain our results.
\end{proof}


\begin{thebibliography}{1}

\bibitem{T2BEWLSV}Dongho Chae, Jiahong Wu, The 2D Boussinesq equations with logarithmically supercritical velocities,
Advances in Mathematics, 230(2012) 1618-1645.
\bibitem{GFD}Pedloshy, J, Geophysical Fluid Dynamics, Springer, New York (1987).
\bibitem{GRFT2BEWPVT}Dongho Chae, Global regularity for the 2D Boussinesq equations with partial viscosity terms,
Advances in Mathematics, 203(2006)497-513.
\bibitem{GWPFABNSSWCD}T. Hmidi, S. Keraani, F. Rousset, Global well-posedness for a Boussinesq-Navier-Stokes system
with critical dissipation, Journal of Differential Equations, 249(2010)2147-2174.
\bibitem{GWPFEBSWCD}T. Hmidi, S. Keraani, F. Rousset, Global well-posedness for Euler-Boussinesq system with
critical dissipation, Communications in Partial Differential Equations, 36(2011)420-445.
\bibitem{OTGWPOACOBNSS}C. Miao, L. Xue, On the global well-posedness of a class of Boussinesq-Navier-Stokes systems,
Nonlinear Differential Equations and Applications NoDEA, 18(2011)707-735.
\bibitem{GWPFT2BSWTTDVATD}C. Wang, Z. Zhang, Global well-posedness for the 2D Boussinesq system with the temperature-dependent
viscosity and thermal diffusivity, Advances in Mathematics, 228(2011)43-62.
\bibitem{GWPFTEBSWAD}T. Hmidi, F. Rousset, Global well-posedness for the Euler-Boussinesq system with axisymmetric data,
Journal of Functional Analysis, 260(2011)745-796.
\bibitem{IBVPFTDVBE}M. Lai, R. Pan, K. Zhao, Initial boundary value problem for two-dimensional viscous Boussinesq equations,
Arch.Rational Mech.Anal. 199(2011)739-760.
\bibitem{GWPFT2IBSWFDAYTD}G. Wu, L. Xue, Global well-posedness for the 2D inviscid B$\acute{e}$nard system with fractional
diffusivity and Yudovich's type data, Journal of Differential Equations, 253(2012)100-125
\bibitem{FANPDE}Bahouri,H., Chemin,J.-Y., Danchin,R. Fourier Analysis and Nonlinear Partial Differential Equations,Grundlehren Math. Wiss.,vol 343, Springer, (2011).
\bibitem{PIF}Chemin,J.-Y. Perfect Incompressible Fluids, Clarendon Press, Oxford, (1998).
\bibitem{T.Hmidi2010JDE}Hmidi,T., Keraani,S., Rousset,F. Global well-posedness for a
Boussinesq-Navier-Stokes system with critical dissipation, J.Diff.Eqs. 249(2010)2147-2174.
\bibitem{AMPATTQGW}A. C$\acute{o}$rdoba, D. C$\acute{o}$rdoba, A maximum principle applied to quasi-geostrophic equations,
Commun. Math. Phys, 249(2004)511-528.
\bibitem{OTGWPOTTDBSWAZV}T. Hmidi, S. Keraani, On the global well-posedness of the two-dimensional Boussinesq system with a zero viscosity,
Indiana Univ.Math.J. 58(4)1591-1618.
\bibitem{GWPOTCBEICBS}C. Miao, G. Wu, Global well-posedness of the critical Burgers equation in critical Besov spaces,
J.Differ.Equ. 247(2009)1673-1693.
\bibitem{TMPATGAFTDQE} N. Ju, The maximum principle and the global attractor for the dissipative 2D quasi-geostrophic
 equations. Commun. Math. Phys. 255(2005)161-181.
\bibitem{ANBIAT2DQGE}Q. Chen, C. Miao, Z. Zhang, A new Bernstein's inequality and the 2D dissipative quasi-geostrophic equation,
Commun. Math. Phys. 271(2007)821-838.

\end{thebibliography}
\end{document}